\documentclass{article}
\usepackage[utf8]{inputenc}

\usepackage{subcaption}
\usepackage{amsmath}
\usepackage{amssymb}
\usepackage{amsthm}
\usepackage{graphicx}
\usepackage{mathtools}
\usepackage{url}
\usepackage{enumerate}
\usepackage{enumitem}
\usepackage{color}
\usepackage[hidelinks]{hyperref}
\usepackage{soul}
\usepackage{cancel}

\usepackage{tikz}
\usepackage{pgfplots}
\usepgfplotslibrary{fillbetween}
\pgfplotsset{compat=1.11}
\usetikzlibrary{patterns}
\usetikzlibrary{calc}

\DeclareRobustCommand{\VAN}[3]{#2} 

\usepackage[margin=1.35in]{geometry}

\def\Lam{\Lambda}
\def\lam{\lambda}

\def\R{\mathbb{R}}

\def\Z{\mathbb{Z}}
\def\eps{\epsilon}

\def\lam{\lambda}

\def\cE{\mathcal E}

\def\cC{\mathcal C}
\def\cP{\mathcal P}

\def\cI{\mathcal I}
\def\cB{\mathcal B}

\def\cX{\mathcal X}
\def\cM{\mathcal M}

\newcommand{\bb}[1]{\mathbb{#1}}
\newcommand{\N}{\bb N}

\newcommand{\ob}[1]{\left(#1\right )} 
\newcommand{\cb}[1]{\left[#1\right ]} 
\newcommand{\set}[1]{\left\{#1\right\}} 
\newcommand{\norm}[1]{\|#1\|} 

\newcommand{\ov}[1]{\overline{#1}}

\newcommand{\ext}{\mathsf{Ext}}

\newcommand{\Ciso}{C_{\text{iso}}}
\newcommand{\proplam}{\rho}

\newtheorem*{theorem*}{Theorem}
\newtheorem{theorem}{Theorem}[section]
\newtheorem{lemma}[theorem]{Lemma}

\newtheorem{prop}[theorem]{Proposition}
\newtheorem*{prop*}{Proposition}

\newtheorem*{claim*}{Claim}
\newtheorem*{fact*}{Fact}
\newtheorem{remark}{Remark}
\newtheorem*{remark*}{Remark}
\newtheorem*{defn*}{Definition}

\newtheorem{assumption}{Assumption}

\theoremstyle{definition}
\newtheorem{example}{Example}

\newcommand{\abs}[1]{\left|#1\right|}

\newcommand{\binf}{\partial_{\text{ext}}}
\newcommand{\odd}{\textsf{o}}
\newcommand{\even}{\textsf{e}}
\newcommand{\inte}{\mathsf{Int}}
\newcommand{\cCb}{\overline{\cC}}
\newcommand{\cU}{\mathcal{U}}

\usetikzlibrary{shapes}
\tikzset{ even/.style={circle,fill=white!25,draw,thick,color=black!75, minimum size=4pt,
    inner sep=0}}

\tikzset{ ovE/.style={circle,fill=lightgray!50,draw,thick, minimum size=15pt,
    inner sep=0}}
\tikzset{ ovO/.style={circle,fill=black!50,draw,thick, minimum size=15pt,
    inner sep=0}}
\tikzset{ evE/.style={circle,fill=white!50,draw,thick,color=gray!50, minimum size=4pt,
    inner sep=0}}
\tikzset{ evO/.style={circle,fill=white!50,draw,thick,color=black, minimum size=4pt,
    inner sep=0}}
\tikzset{ eedge/.style={decorate, decoration={coil,segment
      length=5pt,aspect=0},black}}

\title{Pirogov--Sinai Theory Beyond Lattices}
\author{Sarah Cannon\thanks{Department of Mathematical Sciences, Claremont McKenna College, scannon@cmc.edu}, Tyler Helmuth\thanks{Department of Mathematical Sciences, University of Durham, tyler.helmuth@durham.ac.uk}, Will Perkins\thanks{School of Computer Science, Georgia Institute of Technology, wperkins3@gatech.edu}}

\date{\today}

\begin{document}

\maketitle

\begin{abstract}
  Pirogov--Sinai theory is a well-developed method for understanding
  the low-temperature phase diagram of statistical mechanics models on
  lattices. Motivated by physical and algorithmic questions beyond the
  setting of lattices, we develop a combinatorially flexible version
  of Pirogov--Sinai theory for the hard-core model of independent
  sets. Our results illustrate that the main conclusions of
  Pirogov--Sinai theory can be obtained in significantly greater
  generality than that of $\Z^{d}$. The main ingredients in our
  generalization are combinatorial and involve developing appropriate
  definitions of contours based on the notion of cycle basis
  connectivity. This is inspired by works of Tim\'{a}r and
  Georgakopoulos--Panagiotis.
\end{abstract}

\section{Introduction}

A fundamental question in mathematical statistical mechanics is
whether or not a given system undergoes a phase transition. This was
first understood in the context of the Ising model, where Peierls
established the existence of a phase transition on $\mathbb{Z}^{d}$,
$d\geq 2$~\cite{peierls1936ising, griffiths1964peierls,
  dobrushin1965existence}.  Peierls' strategy can be applied to a wide
range of models and has become an indispensable tool for proving the
existence of phase transitions. The later development of
Pirogov--Sinai theory greatly expanded the scope of statistical
mechanics systems for which phase transitions can be established, and
allowed for the extraction of much more detailed information about
low-temperature
behavior~\cite{pirogov1976phase,Zahradnik,borgs1989unified}. At this
stage, Pirogov--Sinai theory is a textbook method for the study of
discrete spin systems on $\Z^{d}$~\cite{friedli2017statistical} as
well as more general lattices~\cite{MSS2024}.

Physical intuition, however, suggests that such results should not be
restricted to lattices.  Formalizing this intuition has become an
active line of mathematical research, especially in the context of
percolation and percolation-like models, see,
e.g.,~\cite{BS,babson1999cut,DCGRSY,Raoufi,EasoHutchcroft2023}.  One
of the main goals of this paper is to investigate the generality of
the phase transition phenomenon for models of lattice gases. We do
this in the context of the hard-core model of independent sets in
graphs.  In the setting of lattices, the phase transition for the
hard-core model reflects the breaking of a spatial symmetry, as
opposed to an internal (spin-space) symmetry as occurs for the Ising
and Potts models. Thus one might expect a more delicate interplay
between the geometry of the graph considered and the existence of
phase transitions.

\subsection{The Hard-Core Model: Background and Main Results}
\label{sec:model}
To set the stage we briefly define the hard-core model and discuss
some of what is known.  Given a finite graph $G=(V,E)$ and an
\emph{activity} $\lambda\geq 0$, the \emph{hard-core model} is the
probability measure $\mu_{G,\lam}$ on subsets of $V$ given by
\begin{equation}
  \label{eq:intro}
  \mu_{G,\lam}\cb{I} = \frac{\lambda^{|I|}}{Z_G(\lambda)}1_{I\in\cI}, \qquad
  Z_G(\lambda) = \sum_{I\in \cI}\lambda^{|I|}
\end{equation}
where $\cI$ is the set of \emph{independent sets} of $G$, i.e.,
$I\in \cI$ if and only if no two vertices of $I$ are contained in an
edge of $G$ together.  Dobrushin proved one of the first important
theorems about the hard-core model when he established the existence
of a phase transition for the hard-core model on $\Z^{d}$, $d\geq 2$
by a Peierls-type argument~\cite{dobrushin1968problem}.

Dobrushin's result is fairly intuitive. Formally, it is most
  easily described in terms of Gibbs measures $\mu_{G,\lam}$ for infinite graphs $G$, which can be thought of as infinite-volume limits of  the measures
  in~\eqref{eq:intro}. We recall the details about these infinite-volume Gibbs measures in
  Section~\ref{sec:md-infinite}.

When $\lambda$ is small, occupied vertices are sparse and
decorrelated and there is a unique infinite-volume hard-core measure
$\mu=\mu_{\lam}$ on $\Z^d$. The infinite-volume
measure is translation invariant, i.e., 
$\mu(v\in I) = \mu(v'\in I)$ for all $v, v' \in V$. When $\lambda$ is
large, however, multiple infinite-volume Gibbs measures exist, and we say
\emph{phase coexistence} occurs. In particular there are two extremal
Gibbs measures that favor the even and odd sublattices respectively. 
More precisely, Dobrushin proved that there exist infinite-volume Gibbs measures
$\mu^{\even}$ and $\mu^{\odd}$ with the property that
$\mu^{\even}(v\in I)>\mu^{\even}(v'\in I)$ if $v$ is an even vertex
and $v'$ is an odd vertex, and vice-versa for $\mu^{\odd}$. These
measures are thus distinct, and are not translation
invariant. The measures $\mu^{\even}$ and $\mu^{\odd}$ can be constructed as limits of finite-volume measures as in~\eqref{eq:intro} but with even/odd boundary conditions; see Section~\ref{sec:md} for a precise description.

Dobrushin's proof made crucial use of the invariance of $\Z^{d}$ under
lattice shifts. Compared to Peierls spin-flip strategy for the Ising
model, this reflects the fact that the phase transition for the
hard-core model breaks the bipartite spatial symmetry of $\Z^{d}$,
while the phase transition for the Ising model breaks the internal
$\Z_{2}$ (spin space) symmetry of the model. This is not just a matter
of proof technique, as there are well known examples of (transitive)
$d$-dimensional graphs in which the hard-core model does not undergo a
phase transition~\cite{heilmann1972theory,van1999absence}.

Subsequent work on the hard-core model on $\Z^d$ has produced improved
bounds on the values of $\lam$ for which uniqueness of Gibbs measure
holds and for which phase coexistence
holds~\cite{galvin2004phase,restrepo2013improved,blanca2016phase}. Despite
this progress, it remains an open problem to show that for each
$d \ge 2$ there is a unique transition in $\lam$ between uniqueness
and phase coexistence.  In this paper we will focus on results like
that of Dobrushin -- proving the existence of a phase transition in
the hard-core model -- on much more general classes of bipartite
graphs.

Let $G=(V,E)$ be a countably infinite bipartite graph. We will always
assume that $G$ is connected and of bounded maximum degree. In analogy
with $\Z^{d}$, call the two parts of the bipartite graph \emph{even}
and \emph{odd}.  Bipartiteness alone is not enough to imply phase
coexistence will occur for some $\lambda$.  It is known via the FKG
inequality, however (see~\cite[Lemma~3.2]{van1994percolation} and
Section~\ref{sec:md} below), that phase coexistence occurs on
bipartite graphs if and only if all-even and all-odd boundary
conditions lead to distinct infinite-volume measures.  Thus a first
motivation for our work is to understand structural conditions on
bipartite graphs that ensure these extremal boundary conditions lead
to distinct measures.  A second motivation is to understand the broken
symmetry phase of the hard-core model when it exists. We have in mind
both statistical mechanical questions (decay of correlations,
construction of the phase diagram) as well as closely related
algorithmic questions (can one efficiently sample a configuration from
the model on a finite graph). We discuss these motivations, and how
they have lead us to develop a combinatorially flexible version of
Pirogov--Sinai theory, after stating our results.

A graph is \emph{vertex transitive} if for every $u,v\in V$, there is
an automorphism $\pi$ of $G$ such that $\pi(u)=v$. Intuitively, vertex
transitivity is a very strong way of saying the two partition sets of
a bipartite graph look the same, and hence one might expect phase
coexistence at large activities.  To make this precise, recall that a
graph $G$ is \emph{one-ended} if for any finite $S\subset V$,
$G\setminus S$ has a single infinite component. Roughly speaking, this
indicates that the topology of $G$ is similar to that of $\Z^{d}$ for
$d\geq 2$.

A \emph{cycle basis} for $G$ is a generating set $\mathcal{B}$ for the
cycle space of $G$; see Section~\ref{sec:cycle-space-cycle} for a
precise definition. We will assume that the cycle basis is bounded.
That is, for any edge $e$, the number of edges $e' \ne e$ that are in
a basis cycle with $e$ is at most $D=D(\mathcal{B})<\infty$.  Our
first main result is that these conditions suffice for the existence
of a phase transition.

\begin{theorem}
  \label{thm:transitive}
  Suppose $G$ is bipartite, vertex transitive with degree $\Delta$,
  one-ended, and has a $D$-bounded cycle basis $\mathcal{B}$.  There
  exists a $\lambda_{\star}(D,\Delta)<\infty$ such that if
  $\lambda>\lambda_{\star}(D,\Delta)$ then phase coexistence occurs
  for the hard-core model on $G$.
\end{theorem}
  
\begin{remark}
  \label{rem:transitive}
  Infinite connected vertex-transitive graphs have one, two, or
  infinitely many ends. The combinatorial underpinnings of our methods
  are restricted to one-ended graphs. 
\end{remark}
\begin{remark}
  \label{rem:non-ext}
  Generalizations of Theorem~\ref{thm:transitive} cannot include
  two-ended graphs (e.g., $\mathbb{Z}$): the one-dimensional nature of
  such graphs precludes a phase transition from
  occurring. It is also not possible to replace vertex transitive by
  vertex quasitransitive, see Section~\ref{sec:examples},
  Example~\ref{ex:BHW} below.
\end{remark}
\begin{remark}
  \label{rem:infinite-ended}
  Infinitely-ended graphs can have phase transitions, e.g., consider
  the $\Delta$-regular tree. In this case surface effects are
  comparable to boundary effects. This makes proving the existence of
  a transition simpler, and can also lead to new
  phenomena~\cite{lyons2000phase}.
\end{remark}

Our methods do not require an assumption as strong as vertex
transitivity to obtain phase coexistence, and we now give some
alternative hypotheses that suffice. Write
$V=V_{\even}\sqcup V_{\odd}$ for the bipartition of $V$ into the even
and odd parity classes. Call an automorphism $\pi$ of $G$
\emph{matched} if $\{ \{v, \pi(v)\}_{v\in V_{\even}} \}$ is a perfect
matching of $G$.  Intuitively the existence of such an automorphism
captures that the even and odd sides of $G$ look the same. We say that
$G$ is \emph{matched automorphic} if it possesses a matched
automorphism. To establish phase coexistence in the matched
automorphic setting we require one further hypothesis, that the
isoperimetric profile $\Phi_{G}(t)$ of $G$ satisfies
$\Phi_{G}(t)\geq \Ciso \log (t+1) / t$ for some $\Ciso>0$. This
assumption is a mild quantitative assertion that $G$ is not
one-dimensional; see Section~\ref{sec:ends-boundaries} for the
definition of $\Phi_{G}$. Note that this bound on $\Phi_G$ holds
automatically in the setting of Theorem~\ref{thm:transitive}.

\begin{theorem}
  \label{thm:matched}
  Suppose $G$ is bipartite, matched automorphic, one-ended, of maximum
  degree $\Delta$, and has a $D$-bounded cycle basis
  $\mathcal{B}$. Suppose further that
  $\Phi_{G}(t)\geq \Ciso \log (t+1)/t$.  There exists
  $\lambda\geq\lambda_{\star}(D,\Delta,\Ciso)<\infty$ such that if
  $\lambda\geq\lambda_{\star}(D,\Delta,\Ciso)$ then phase coexistence
  occurs for the hard core model on $G$.
\end{theorem}

Our methods also allow us to understand typical high density
independent sets on bipartite graphs lacking any symmetry between the
sides of the bipartition.  Call a bipartite graph $G$ \emph{vertex
  transitive within each parity class} if for any two vertices $u,v$
in the same parity class (e.g., $u,v\in V_{\even}$) there is an
automorphism $\pi$ of $G$ with $\pi(u)=v$.  For such graphs we
consider the hard-core model with parity-dependent activities
$\lambda_{\even}$ and $\lambda_{\odd}$, in accordance with the fact
that there is no symmetry between the even and odd parity classes. The
partition function is
  \begin{equation*}
    Z_G(\lam_\even,\lam_\odd) =
    \sum_{I \in \mathcal I(G)} \lam_\even^{| I \cap V_\even|} \lam_\odd ^{| I \cap V_\odd|} \,. 
\end{equation*}

On $\mathbb{Z}^{d}$, intuition suggests that a discrepancy between
$\lam_{\even}$ and $\lam_{\odd}$ will lead to a unique Gibbs measure
when $d\geq 2$: if there are enough occupied vertices for a parity
class to be preferred, the class with larger activity will be vastly
preferred. Proving this for all possible activities is largely
open~\cite{van1994percolation,haggstrom1997ergodicity}. When there is
no symmetry between the two parity classes, obtaining a coexistence
result intuitively requires $\lam_{\even}$ and $\lam_{\odd}$ to be
delicately balanced, as the density of the largest independent sets
may not be equal, and more subtly, as the entropy from defects to the
all-even and all-odd configuration are no longer equal.  Establishing
this balance is precisely what Pirogov--Sinai theory was built to
achieve, and our next results show that the combinatorially flexible
version of Pirogov--Sinai theory developed in this paper has the power
to carry this out.

Let $B_{k}(v)$ denote the ball of radius $k$ about a vertex $v$, i.e.,
$B_{k}(v) = \{ w\in V \mid d_{G}(v,w)\leq k\}$, with $d_{G}(v,w)$ the
graph distance between $v$ and $w$ in $G$.  Given a vertex $v$, define
the \emph{free energy} (or \emph{pressure}) of the hard-core model on
$G$ by
\begin{equation}
  \label{eq:fg}
  f_G(\lam_\even,\lam_\odd) = \lim_{n \to \infty} \frac{1}{|E(B_n(v))|}
  \log Z_{B_n(v)} (\lam_\even,\lam_\odd) \,.
\end{equation}
When this limit exists one expects the points of non-analyticity of
the free energy to coincide with points of non-uniqueness.  To ensure
this limit does exist, we will assume that $G$ has at most
\emph{polynomial volume growth}, i.e., there exist $c,\alpha>0$ such
that $|B_{n}(v)|\leq cn^{\alpha}$ for all $n\geq 1$ and all vertices
$v$. The assumption of vertex transitivity within a class ensures the
limit in~\eqref{eq:fg} does not depend on the choice of vertex $v$.

It is natural to parametrize $\lam_{\odd}$ in terms of $\lam_{\even}$
by setting
$\lam_{\odd}=\proplam \lam_{\even}^{\Delta_{\odd}/\Delta_{\even}}$ for
$\proplam\in (\frac12,2)$. Intuitively, this parametrization accounts
for the difference in the density of the all-even and all-odd
independent sets when $\Delta_{\odd}\ne \Delta_{\even}$.  The
restriction of $\proplam$ to $(\frac12,2)$ is largely arbitrary and
what is important is that this interval contains $\proplam=1$.  Write
$\cU$ for the set of possible $\lam_{\odd}$ in this parametrization.
Our main result for graphs that are vertex transitive within a class
follows; the isoperimetric constant $\Ciso$ in the statement exists
and is finite due to the assumptions on $G$, see
Section~\ref{sec:ends-boundaries}.
\begin{theorem}
  \label{thm:biregular}
  Suppose $G$ is bipartite, one-ended, vertex transitive within each
  parity class, and has polynomial volume growth. Suppose also that
  $G$ has a $D$-bounded cycle basis $\cB$. There is a
  $\lam_{\star}=\lam_{\star}(D,\Delta,\Ciso)$ such that if
  $\lambda_{\even} \geq\lambda_{\star}$ and $\lam_\odd\in \cU$, then
  the free energy $f_G(\lam_\even,\lam_\odd)$ is well-defined and
  independent of the vertex $v$ used in its definition. Moreover,
  there is a unique $\lambda_{\odd,c}(G,\lambda_{\even})\in \cU$ such
  that
  \begin{enumerate}
  \item The function $f_G(\lam_\even,\lam_\odd)$ is continuously
    differentiable at pairs of activities
    $(\lam_\even,\lam_\odd )$ different from 
    $(\lam_\even,\lam_{\odd,c}(G,\lam_\even))$.
  \item The critical parameter $\lambda_{\odd,c}$ satisfies
    $\log\lam_{\odd,c} = \frac{\Delta_{\odd}}{\Delta_{\even}}
    \log\lam_{\even} + o_{\lam_{\even}}(1)$.
  \item If $\lambda_{\odd}\neq \lambda_{\odd,c}$ there is a unique
    infinite-volume Gibbs measure.
  \item For the hard-core model on $G$ with activities
    $(\lam_\even,\lam_{\odd,c}(G,\lam_\even))$ phase coexistence
    occurs, and $f_G$ fails to be continuously differentiable at such
    pairs.
  \end{enumerate}
\end{theorem}
At the point $\proplam=1$, the fully occupied even and odd
configurations have equal weight, and hence one expects phase
coexistence to occur for $\proplam\approx 1$. The second conclusion of
the theorem shows this intuition is correct as
$\lam_\even\uparrow\infty$.  Our methods in fact give more detail:
they provide a convergent series representation of
$\log \lambda_{\odd,c} - \frac{\Delta_{\even}}{\Delta_{\odd}}
\log\lambda_{\even}$. By computing terms of this series the equation
of the coexistence curve $\lambda_{\odd,c}(G,\lambda_{\even})$ can be
computed to arbitrarily high accuracy.

In Section~\ref{sec:examples} below we illustrate
Theorems~\ref{thm:transitive}--~\ref{thm:biregular} via several
examples. Before this, however, we state our algorithmic results for
finite graphs. Some related results and context are reviewed in
Sections~\ref{sec:sm} and~\ref{sec:alg}, and the methods behind our
results are discussed in Section~\ref{sec:pi}.

\subsection{Algorithmic Results}
\label{subsecAlgorithms}

There are two natural computational problems associated to a
statistical physics model like the hard-core, Ising, or Potts model on
finite graphs. The first is to compute (or approximate) the partition
function $Z_G$; the second is to (approximately) sample from the Gibbs
measure $\mu_G$.  In general, computing partition functions exactly is
$\#\text{P}$-hard (as hard as computing the number of satisfying
assignments to a boolean satisfiability formula), while the
tractability of approximating partition functions (approximate
counting) depends on the specific model, the parameters, and the class
of input graphs considered.  An approximate counting algorithm is
considered efficient if it outputs an $e^\eps$ multiplicative
approximation to $Z_G$ and runs in time polynomial in the size of $G$
and $1/\eps$. An approximate sampling algorithm is efficient if its
output distribution is within $\eps$ total variation distance of
$\mu_G$, and runs in time polynomial in the size of $G$ and $1/\eps$.
See~\cite{jerrum1986random,
  jerrum2003counting,sinclair1989approximate} and
Section~\ref{sec:alg} below for more details on approximate counting
and sampling.

While approximate counting and sampling for the hard-core model are
NP-hard in general, the complexity of the two tasks for bipartite
graphs is unknown~\cite{dyer2004relative} and defines the complexity
class $\#\text{BIS}$ (counting bipartite independent set).

Our methods give efficient algorithms for a subclass of
$\#\text{BIS}$, instances that arise by imposing boundary conditions
on finite subgraphs of the infinite graphs considered in
Theorems~\ref{thm:transitive}--~\ref{thm:biregular}.

Given a finite subgraph $H$ of $G$, there is a natural notion of even
and odd boundary conditions for the hard-core model on $H$ induced by
a given cycle basis $\cB$ of $G$ (see
Section~\ref{sec:boundary-conditions}), and of the associated
finite-volume partition functions $Z^{\odd}_{H}(\lam)$ and
$Z^{\even}_{H}(\lam)$.  Our implementation of a combinatorial
Pirogov--Sinai theory gives strong analytic control (convergent
cluster expansions) for these partition functions. These convergent
expansions in turn lead to efficient algorithms for approximating the
hard-core partition functions.  This consequence was an important
motivation for this work.
 
To state our algorithmic results precisely, recall that a \emph{fully
  polynomial-time approximation scheme (FPTAS)} for
$Z^{\odd}_{H}(\lambda)$ is a deterministic algorithm that produces a
number $\hat Z^{\odd}_{H}(\lambda)$ such that
$e^{-\epsilon}Z^{\odd}_{H}(\lambda)\leq \hat Z^{\odd}_{H}(\lambda)\leq
e^{\epsilon}Z^{\odd}_{H}(\lambda)$ and runs in time polynomial in
$n/\epsilon$, where $n = |V(H)|$ and $\eps >0$ is an error tolerance
(see, e.g.,~\cite{jerrum2003counting}). If $\mu^{\odd}_{H}$ is the
probability distribution of the hard-core model on $H$ with odd
boundary conditions, then an \emph{efficient approximate sampling
  algorithm} for $\mu^{\odd}_{H}$ is a randomized algorithm that
outputs an independent set with distribution $\hat\mu^{\odd}_{H}$ such
that $\|\mu^{\odd}_{H}-\hat \mu^{\odd}_{H}\|_{\text{TV}}\leq \epsilon$
and runs in time polynomial in $n/\epsilon$, where
$\| \cdot - \cdot\|_{TV}$ denotes the total variation distance.

Formally, one asks about the existence of algorithms satisfying these
guarantees for an infinite collection $\mathcal H$ of finite
graphs. For instance, $\mathcal H$ could be the set of all finite
graphs, or the set of finite graphs of maximum degree at most
$\Delta$.  Our results will apply to sets of finite subgraphs of an
infinite graph $G$ satisfying the conditions of our earlier
theorems. Given a finite subgraph $H$ of $G$, we write $\partial H$
for the edge boundary
$\{ \{v,w\}\in E(G) \mid |\{v,w\}\cap V(H)|=1\}$.

\begin{theorem}
  \label{thm:expansion}
  Suppose $G$ satisfies the conditions of
  Theorem~\ref{thm:transitive}, or satisfies the conditions of
  Theorem~\ref{thm:matched} except possibly the condition on the
  isoperimetric profile.  Let $\mathcal{H}$ be the set of induced
  subgraphs $H\subset G$ that are finite, connected, and have all
  vertices in $V(H)\cap\partial H$ having the same parity. There is a
  $\lambda_{\star}>0$ such that for $\lambda>\lambda_{\star}$ there is
  an FPTAS to compute $Z^{\even}_{H}(\lambda)$ and
  $Z^{\odd}_{H}(\lambda)$ for $H\in\mathcal{H}$ and efficient
  approximate sampling algorithms for $\mu^{\even}_{H}$ and
  $\mu^{\odd}_{H}$ for $H\in\mathcal{H}$.
\end{theorem}

It is straightforward to extend Theorem~\ref{thm:expansion} to the
setting of Theorem~\ref{thm:biregular} at
$\lambda_{\odd} = \lambda_{\odd,c}$ where phase coexistence
occurs. Further technical work would be required to obtain algorithms
away from the coexistence point; see Section~\ref{sec:algorithms}.

\subsection{Illustrative Examples}
\label{sec:examples}

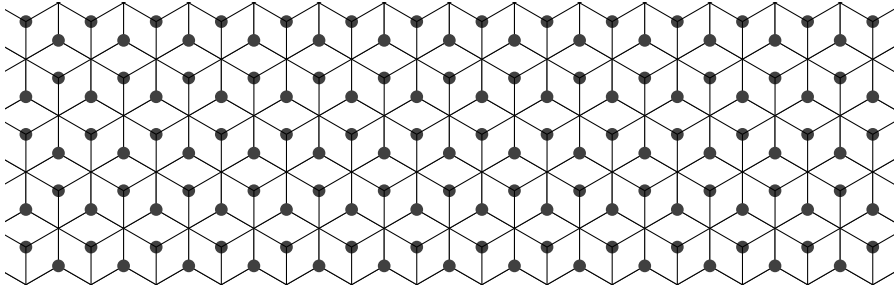
\begin{figure}[t]
  \centering
  \scalebox{1}{
    \begin{tikzpicture}[hexa/.style= {shape=regular polygon,regular
      polygon sides=6,minimum size=1cm, draw,fill=white,inner
      sep=0,rotate=30}]

    \begin{scope}
    \clip (-2,4) rectangle (9.8,7.75);
      
    \foreach \j in {0,...,10}{%
      \pgfmathsetmacro\end{10+\j} 
    \foreach \i in {0,...,\end}{%
    \node[hexa] (h\i;\j) at ({(\i-\j/2)*sin(60)},{\j*0.75}) {};
    \node[even] at ({(\i-\j/2)*sin(60)},{\j*0.75}) {};
    \draw[black] (h\i;\j) ++(0,0) -- +(30:.5cm);
    \draw[black] (h\i;\j) ++(0,0)-- +(150:.5cm);
  \draw[black] (h\i;\j) ++(0,0)-- +(270:.5cm);}}
\foreach \j in {0,...,9}{%
  \pgfmathsetmacro\end{19-\j} 
  \foreach \i in {0,...,\end}{%
  \pgfmathtruncatemacro\k{\j+11}  
  \node[hexa] (h\i;\k) at ({(\i+\j/2-4.5)*sin(120)},{8.25+\j*0.75}) {};
  \node[even] at ({(\i+\j/2-4.5)*sin(120)},{8.25+\j*0.75}) {};
  \draw[black] (h\i;\k) ++(0,0) -- +(30:.5cm);
    \draw[black] (h\i;\k) ++(0,0)-- +(150:.5cm);
    \draw[black] (h\i;\k) ++(0,0)-- +(270:.5cm);}}

\begin{scope}[yshift=.5 cm]
    \foreach \j in {0,...,10}{%
      \pgfmathsetmacro\end{10+\j} 
    \foreach \i in {0,...,\end}{%
    \node[even] at ({(\i-\j/2)*sin(60)},{\j*0.75}) {};
  }}
\foreach \j in {0,...,9}{%
  \pgfmathsetmacro\end{19-\j} 
  \foreach \i in {0,...,\end}{%
  \pgfmathtruncatemacro\k{\j+11}  
  \node[even] at ({(\i+\j/2-4.5)*sin(120)},{8.25+\j*0.75}) {};
  }}
\end{scope}
\end{scope}
\end{tikzpicture}}
\caption{A portion of the dice lattice. The vertices in one
  bipartite class are emphasized. 
}
  \label{fig:DL}
\end{figure}

The following examples show the flexibility of our results, as well as
some of their limitations.  For the examples showing phase coexistence
with constant $\lambda$, direct (case-by-case) Peierls-type arguments
are likely possible. Our approach yields stronger results and requires
checking only a few very simple hypotheses. It seems unlikely a direct
Peierls-type argument could establish coexistence as shown in
Example~\ref{ex:biregular}.

\begin{example}
  \label{ex:Zd}
  Theorem~\ref{thm:transitive} applies immediately to $\mathbb{Z}^{d}$
  for $d\geq 2$. This re-proves Dobrushin's result that phase
  coexistence occurs for $\lambda$ large enough. In the bivariate
  activity setting, Theorem~\ref{thm:biregular} implies that there is
  a unique Gibbs measure for large but unequal activities (more
  precisely, $\lam_\even\geq \lam_\star$ and $\lam_\odd\in \cU$).
\end{example}

\begin{example}
  \label{ex:biregular}
  The \emph{dice lattice} is the infinite one-ended bipartite graph a
  portion of which is shown in Figure~\ref{fig:DL}. Even vertices have
  degree three, while odd vertices have degree six. Write
  $\lambda_{3}=\lambda_\even$ and $\lambda_{6}=\lambda_\odd$.  By
  Theorem~\ref{thm:biregular}, if $\lambda_{3}\geq\lambda_{\star}$
  then there is a value of $\lambda_{6}$ satisfying
  $\log \lambda_{6}=2\log\lambda_{3}+o_{\lam_{3}}(1)$ such that there
  is phase coexistence. For any other value of $\lambda_{6}\in \cU$,
  however, there is uniqueness. In fact,
  $\log\lam_{6} = 2\log\lam_{3} + 6\lambda_{3}^{-1} +
  o(\lam_{3}^{-1})$, see Section~\ref{sec:phases},
  Remark~\ref{rem:rhomb}.
\end{example}

\begin{example}
  \label{ex:perc}
  Consider the graph $\Z^{d}\times \{0,1\}$ in which two copies of
  $\Z^{d}$ are stacked on top of one another and connected by vertical
  edges. That is, for each $x\in \Z^{d}$, the vertex
  $(x,i)\in\Z^{d}\times\{0,1\}$ is connected to $(y,i)$ for
  $\{x,y\}\in E(\Z^{d})$, and further $(x,0)$ is also connected to
  $(x,1)$ for all $x\in \Z^{d}$. Elementary plaquettes give a bounded
  cycle basis. Provided $d\geq 2$, this graph is one-ended and
  satisfies $\Phi_{\Z^{d}\times \{0,1\}}(t)\geq \Ciso t^{-1/2}$, see
  Lemma~\ref{lem:iso}. There is a matched automorphism $\pi$ given by
  matching $(x,0)$ to $(x,1)$. Theorem~\ref{thm:matched} thus implies
  that phase coexistence occurs on this graph.

  For $X\subset \Z^{d}$ note that $\pi$ continues to be a matched
  automorphism of the graph $G$ obtained from $\Z^{d}\times \{0,1\}$
  by removing the vertices $\{ (x,i) \mid x\in X, i=0,1\}$. More
  precisely, $\pi$ restricted to the vertices of $G$ is a matched
  automorphism.  If the connected components of $X$ (as an induced
  subgraph of $\Z^{d}$) are uniformly bounded in size, then $G$ has a
  bounded cycle basis. Since $G$ is rough isometric to
  $\Z^{d}\times \{0,1\}$, $\Phi_{G}(t)\geq \Ciso' t^{-1/2}$ for some
  $\Ciso'>0$, see the proof of Lemma~\ref{lem:iso}. Thus $G$ satisfies
  the hypothesis of Theorem~\ref{thm:matched}, and phase coexistence
  occurs for $\lambda\geq \lam_\star$.
\end{example}

\tikzset{ evB/.style={circle,fill=white!50,draw,color=black, minimum size=4pt,
    inner sep=0}}
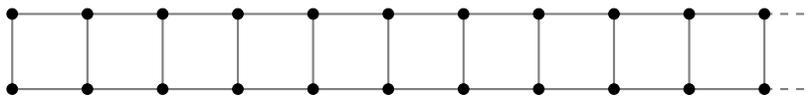
\begin{figure}[h]
  \centering
  \begin{tikzpicture}
    \draw[gray, thick] (0,0) grid (10,1);
    \draw[gray, thick, dashed] (10,0) -- (10.8,0);
    \draw[gray, thick, dashed] (10,1) -- (10.8,1);
    \node[evB] at (0,0) {};
    \node[evB] at (0,1) {};
    \node[evB] at (1,0) {};
    \node[evB] at (1,1) {};
    \node[evB] at (2,0) {};
    \node[evB] at (2,1) {};
    \node[evB] at (3,0) {};
    \node[evB] at (3,1) {};
    \node[evB] at (4,0) {};
    \node[evB] at (4,1) {};
    \node[evB] at (5,0) {};
    \node[evB] at (5,1) {};
    \node[evB] at (6,0) {};
    \node[evB] at (6,1) {};
    \node[evB] at (7,0) {};
    \node[evB] at (7,1) {};
    \node[evB] at (8,0) {};
    \node[evB] at (8,1) {};
    \node[evB] at (9,0) {};
    \node[evB] at (9,1) {};
    \node[evB] at (10,0) {};
    \node[evB] at (10,1) {};
  \end{tikzpicture}
  \caption{A finite portion of the width-two semi-infinite cylinder
    graph.}
  \label{fig:ladder}
\end{figure}
\begin{example}
  \label{ex:cylinder}
  The \emph{width-$k$ (semi-infinite) cylinder graph} is the one-ended
  graph $C_{k}\times \N$, where $C_{k}$ is the cycle on $k$
  vertices. See Figure~\ref{fig:ladder}.  Cylinder graphs with even
  widths are matched automorphic, while cylinder graphs with odd width
  are not.

  The hard-core model does not have a phase transition on cylinder
  graphs, as these graphs are (essentially) one dimensional.
  Nonetheless we obtain a convergent expansions and algorithms for
  \emph{even-width} cylinder graphs by
  Theorem~\ref{thm:expansion}. Phase coexistence does not follow since
  these graphs do not satisfy the expansion condition of
  Theorem~\ref{thm:matched}.

  Our failure to obtain an expansion for odd-width cylinder graphs is
  for good reason: by explicit computation the relevant expansion
  variable for the width-$1$ cylinder graph is $\lambda^{-1/2}$, not
  $\lambda^{-1}$. Perhaps more surprisingly, boundary conditions are
  also relevant: for periodic cylinder graphs of even width there
  cannot be a convergent expansion in powers of $\lam^{-1}$,
  see~\cite[Section~4]{BBPR}.
\end{example}

\begin{example}
  \label{ex:Cayley}
  This example generalizes Example~\ref{ex:Zd}. Let $\mathsf{G}$ be a
  finitely presented group such that (i) $\mathsf{G}$ has an index two
  subgroup $\mathsf{H}\subset \mathsf{G}$, (ii) $\mathsf{G}$ has
  superlinear volume growth, and (iii) $\mathsf{G}$ has at most
  polynomial volume growth. Suppose $h_{1},\dots, h_{n}$ generate
  $\mathsf{H}$, let $g\notin \mathsf{H}$, and set
  $S = \{g, h_{1}g,h_{2}g,\dots, h_{n}g\}$. The Cayley graph $G$
  generated by $S\cup S^{-1}$ is vertex transitive, bipartite, has a
  bounded cycle basis~\cite{timar2007cutsets,timar2013boundary}, has
  bounded degree, and is one-ended (by the assumptions on volume
  growth; see the proof of Lemma~\ref{lem:iso}).  By
  Theorem~\ref{thm:transitive}, there is phase coexistence for the
  hard-core model on $G$ when $\lam$ is sufficiently large. Moreover,
  Theorem~\ref{thm:biregular} implies that phase coexistence does not
  occur if the activities on the two parity classes are unequal (more
  precisely, if $\lam_{\even}$ is sufficiently large and
  $\lam_{\odd}\in\cU$).
\end{example}

\begin{example}
  \label{ex:BHW}
  The following is a non-example. Attach to each vertex $v$ of
  $\Z^{d}$ a binary tree of depth one rooted at $v$; see
  Figure~\ref{fig:BHW}. Following~\cite{BHW} call this graph
  $\Z^{d}_{2}$. For $d\geq 2$ the graph $\Z^{d}_{2}$ is one-ended,
  bipartite, quasitransitive, and has polynomial volume growth, but it
  does not satisfy any of our symmetry assumptions. For $\lambda$
  large enough, there is a unique hard-core Gibbs measure for
  $\Z^{d}_{2}$, see~\cite[Lemma~2.2]{BHW}.
\end{example}

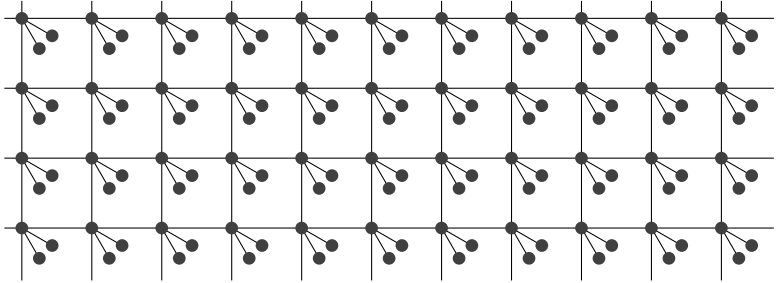
\begin{figure}
  \centering
  \begin{tikzpicture}[scale=.93]
    \draw[white,fill=white] (-.25,-.5) rectangle (10.5,3.5);
    \draw[black] (-.25,-.75) grid (10.75,3.25);
    \foreach \j in {0,...,3}{%
  \foreach \i in {0,...,10}{%
    \draw[black] (\i,\j)  -- +(-30:.5) node[even]{};
    \draw[black] (\i,\j) -- +(-60:.5) node[even]{};
    \node[even] at (\i,\j) {};
  }}
  \end{tikzpicture}
  \caption{A subset of the graph $\Z^2_2$ from Example~\ref{ex:BHW}.}
  \label{fig:BHW}
\end{figure}

\subsection{Context and Related Work: Statistical Mechanics}
\label{sec:sm}
Inspired by conjectures of Benjamini and Schramm~\cite{BS}, there has
been a great deal of interest in understanding the generality in which
Bernoulli bond percolation undergoes a phase transition. This has been
done both by using generalized Peierls-type
arguments~\cite{babson1999cut,timar2007cutsets,georgakopoulos2018analyticity}
as well as other methods~\cite{DCGRSY,EasoHutchcroft2023}. Our use of
a cycle basis was inspired by the work of Georgakopolous and
Panagiotis in their investigation of analyticity properties of bond
percolation~\cite{georgakopoulos2018analyticity}. The utility of a
bounded cycle basis assumption was first realized by
Tim\'{a}r~\cite{timar2007cutsets,timar2013boundary}, who was inspired
by questions posed in~\cite{babson1999cut}.  As remarked
in~\cite{DCGRSY}, phase coexistence results for Bernoulli percolation
can be combined with comparison methods to deduce phase coexistence
for the $q$-state random cluster model with $q\geq 1$. 

The generality of phase coexistence phenomena for other statistical
mechanics models has been less explored. The hard-core model has been
comparatively well-studied, but prior results have relied on strong
lattice-like
assumptions~\cite{JauslinLebowitz,mazel2018high,mazel2019high,mazel2020hard,MSS2024},
strong expansion
conditions~\cite{galvin2006slow,jenssen2020algorithms}, or have
specialized to the study of the hard-core model on
trees~\cite{BHW,galvin2011multistate}.  We note that while we have
only considered the hard-core model, our methods can likely be applied
more broadly without the introduction of significant new ideas. That
is, using the notion of a cycle basis as the combinatorial
underpinning of a generalization of Pirogov--Sinai theory is a general
strategy. Implementing this strategy for spin models should be broadly
similar to the work in this paper, albeit with less subtleties about
the necessary properties of the underlying graph due to the relevant
symmetry being an internal (spin-space) symmetry.

In the setting of $\Z^{d}$ and tori $(\Z/L\Z)^{d}$, Pirogov--Sinai
theory has been used to study finite-size corrections as
$L\to\infty$~\cite{borgs1990rigorous}. For example,
\cite{borgs1990rigorous} defines a natural finite-volume critical
point and determines the discrepancy of this point from the
infinite-volume critical point. It would be interesting if the methods
of the present paper could be extended to investigate finite-size
corrections in greater generally, e.g., for sequences of finite graphs
whose local limits satisfy the hypotheses of our main theorems.
  
There have been recent breakthrough results concerning Bernoulli
percolation on general (finite and infinite) transitive
graphs~\cite{DCGRSY,HutchcroftTointon,EasoHutchcroft,PanagiotisSevero,EasoHutchcroft2023}. These
results have relied on the development of new, non-perturbative
methods not based on contours. The development of non-perturbative
methods for the hard-core model would be very interesting, as would be
extensions to the setting of finite graphs.  For example, building
on~\cite{DCGRSY}, in~\cite{PanagiotisSevero} it was shown that there
is an $\epsilon>0$ such that $p_{c}(G)\leq 1-\epsilon$ for all Cayley
graphs $G$ that are not one-dimensional. In particular, $\epsilon$
does not depend on the degree or a cycle basis condition. An analogous
result for the hard-core model would be very interesting.

Jauslin and He~\cite{JH} have very recently extended the methods
of~\cite{JauslinLebowitz} from lattice-like hard-core models that tile
$\mathbb{R}^{d}$ to lattice-like models that may only partially tile
space, in part inspired
by~\cite{mazel2018high,mazel2019high,mazel2020hard}. Their assumptions
and methods are designed for understanding discretizations of
continuum (Euclidean) models, and are complementary to ours. Our
methods do not apply to their general situation (which may include
non-bipartite graphs), nor do their methods apply in the generality of
the present paper.

\subsection{Context and Related Work: Algorithms}
\label{sec:alg}

When $\lam$ is sufficiently large, Theorem~\ref{thm:expansion} gives
efficient approximate counting and sampling algorithms for the
hard-core model on finite subgraphs of the infinite graphs under
consideration in this paper.  To explain the significance of this
result, we briefly give some background. For any fixed value of
$\lam>0$, the approximate counting and sampling problems for the
hard-core model are computationally hard (NP-hard) in general, i.e.,
when the set of possible input graphs $\mathcal{H}$ consists of
\emph{all} finite graphs.  When restricting the allowed input graphs
to the class of finite graphs of maximum degree $\Delta$, the problems
are computationally hard (no polynomial time algorithms exists unless
NP$=$RP) when
$\lam > \lam_c(\Delta) \approx
\frac{e}{\Delta}$~\cite{sly2010computational,sly2014counting,galanis2011improved},
while efficient algorithms exist when
$\lam <\lam_c(\Delta)$~\cite{weitz2006counting}.  Remarkably, the
critical point $\lam_{c}(\Delta)$ coincides with the activity at which
a phase transition occurs for the hard-core model on the
$\Delta$-regular tree.

Given this, one can ask if there exist structural assumptions on
finite graphs that make the hard-core approximate counting and
sampling problems tractable even when they are hard in the worst-case.
For example, there are efficient algorithms for these problems for the
class of claw-free
graphs~\cite{jerrum1989approximating,matthews2008markov}. This line of
questioning has lead to a significant open problem: do efficient
approximate counting and sampling algorithms exist for the hard-core
model on the class of \emph{bipartite} graphs~\cite{dyer2004relative}?
To date, no efficient algorithms are known, nor is the problem known
to be NP-hard.  Many other approximate counting and sampling problems
with unknown complexity are, however, known to be
\#BIS-hard~\cite{goldberg2007complexity}, that is, as hard as the
problem of approximately counting the number of independent sets in a
bipartite graph. For example, the ferromagnetic Potts model, the Ising
model with arbitrary external fields, and stable matchings are all
known to be \#BIS-hard.  For \emph{bounded degree} bipartite graphs,
approximate counting and sampling in the hard-core model is \#BIS-hard
when $\lam>\lam_c(\Delta)$~\cite{cai2016hardness}.

Intuition suggests that the phase transition phenomenon in bipartite
graphs should make it easier to design algorithms: independent sets
that are mostly even or mostly odd have a relatively simple structure.
This intuition has been made rigorous for some special classes of
graphs: subgraphs of $\Z^d$~\cite{helmuth2020algorithmic} and random
regular and bounded-degree expander
graphs~\cite{jenssen2020algorithms,liao2019counting,chen2022sampling,cannon2020counting}. See
also~\cite{barvinok2017weighted}. The algorithms
of~\cite{helmuth2020algorithmic} for $\Z^d$ made use of Pirogov--Sinai
theory.

Theorem~\ref{thm:expansion} shows that the idea of using
Pirogov--Sinai theory to design algorithms applies far beyond
$\Z^d$. It would be interesting if the methods of this paper could be
extended to finite graphs satisfying appropriate symmetry conditions,
rather than only applying to finite subgraphs of infinite graphs
satisfying symmetry conditions. Furthermore, it would be very
interesting to show \#BIS-hardness for classes of bipartite graphs
satisfying some of the symmetry conditions assumed here, or even for
graphs with bounded cycles bases.

\subsection{Proof Ideas and Outline of Paper}
\label{sec:pi}

The central new idea in this paper is that the notion of a \emph{cycle
  basis} is enough to drive a generalization of the combinatorial
aspects of a Pirogov--Sinai analysis. Let us briefly indicate how this
is carried out. In Section~\ref{sec:contours} our main result is
Proposition~\ref{lem:bijection}, which gives a representation of the
hard-core model on general one-ended bipartite graphs in terms of
\emph{contour models}. Intuitively, which contours separate
even-occupied and odd-occupied regions from one another. Our
definition of contours relies on a given cycle basis. In conjunction
with a one-endedness assumption, there is a natural notion of what the
`outside' of a contour is. This leads to a nested (partially ordered)
structure on the set of contours, and this partial order is crucial in
subsequent steps of our analysis. The relevance of partially ordering
contours was first made explicit in~\cite{borgs2012tight}, which used
tools from algebraic topology in the context of $d$-dimensional
tori. Prior works had relied on topological properties of $\R^{d}$. By
contrast, our approach is purely combinatorial. The use of a cycle
basis to replace topological argument by combinatorial arguments was
first recognized by Tim\'{a}r, and we have drawn inspiration from his
work and that of
Georgakopoulos--Panagiotis~\cite{timar2007cutsets,georgakopoulos2018analyticity}.

Carrying out a Pirogov--Sinai analysis requires further ingredients to
obtain analytic control of partition functions. It is for this reason
that our main results assume the existence of a \emph{bounded} cycle
basis. More significantly, we require a ``Peierls estimate'' that
measures the cost of a contour. This requires further structural
assumptions beyond being bipartite and one-ended. Sufficient criteria
are developed in Section~\ref{sec:peierls} These criteria are
hypotheses on the existence of an appropriate spatial symmetry. As
discussed in the next section, the determination of more general
sufficient criteria would be interesting.

Given the above, in Section~\ref{sec:convergence} we show how the
analytic steps in Pirogov--Sinai theory can be adapted to gain the
desired control of partition functions. One significant difference
compared to $\mathbb{Z}^{d}$ is that control of the bulk free energy
of the model is not immediate; on $\mathbb{Z}^{d}$ this is a
consequence of vertex transitivity. The symmetry assumptions we have
considered in obtaining Peierls estimates are, however, enough to
obtain the needed control.

Given analytic control (i.e., convergent series expansions), the
derivation of statistical mechanical and algorithmic consequences
follows from arguments familiar from the lattice setting. This is done
in Section~\ref{sec:applications}.

\section{Preliminaries}
\label{sec:background}

\subsection{Graph Notation and Terminology}
\label{sec:graphterm}

A graph $G=(V,E)$ is \emph{bipartite} if there is a partition
$V=V_{\odd}\sqcup V_{\even}$ of the vertices such that all edges
contain exactly one \emph{odd} vertex in $V_{\odd}$ and one
\emph{even} vertex in $V_{\even}$. We sometimes refer to even and odd
as the \emph{parity} of a vertex. An edge $e=\{x,y\}$ is
\emph{incident} to a vertex $v$ if $v=x$ or $v=y$, and we sometimes
call $x,y$ the \emph{endpoints} of $e$. Vertices $x$ and $y$ are
\emph{adjacent} if they appear in an edge together. A \emph{path} is a
sequence of adjacent vertices. A path is \emph{simple} if no vertex is
repeated. A path is a \emph{cycle} if it begins and ends at the same
vertex, no other vertex is repeated, and it has length at least
three. The \emph{length} of a path or cycle is the number of edges it
contains. We write $|C|$ for the length of a cycle $C$.

Given a finite set $H\subset V(G)$ the subgraph \emph{induced} by $H$
has edge set $E(H) = \{ \{x,y\} \in E(G) \mid x,y\in H\}$. $H$ is
\emph{connected} if the subgraph induced by $H$ is connected. For a
collection of edges $E' \subseteq E(G)$, $G \setminus E'$ will denote
the graph $G'$ with vertex set $V(G)$ and edge set
$E(G) \setminus E'$.

An \emph{independent set} of a graph $G$ is a subset of vertices no
two of which are adjacent. The set of independent sets of $G$ is
denoted $\cI(G)$.

An \emph{automorphism} of a graph $G=(V,E)$ is a bijective map
$\pi\colon V\to V$ that maps edges to edges, i.e., such that
$\{u,v\}\in E$ if and only if $\{\pi(u),\pi(v)\}\in E$. A graph is
\emph{vertex transitive} if for any two vertices $v_1$ and $v_2$,
there is an automorphism $\pi$ with $\pi(v_1) = v_2$. A graph is
\emph{vertex transitive within each parity class} if for any
$v_{1},v_{2}\in V_{\odd}$ or $v_{1},v_{2}\in V_{\even}$ there is an
automorphism $\pi$ with $\pi(v_{1})=v_{2}$.  A graph is \emph{matched
  automorphic} if it possesses a \emph{matched automorphism}, meaning
an automorphism $\pi$ such that
$\{ \{v,\pi(v)\} \}_{v\in V_{\odd}}\subset E$ is a perfect matching of
$G$, i.e., a subset of edges such that every vertex is contained in
exactly one such edge.

\subsection{Ends, Boundaries, and Isoperimetry}
\label{sec:ends-boundaries}

An infinite graph $G$ is \emph{one-ended} if for any finite vertex set
$S$, $G \setminus S$ has only one infinite component. We will only
consider one-ended infinite graphs in this paper.

Let $\Lambda \subseteq G$ be a finite subgraph of $G$.  Let
$\partial \Lambda$ denote the \emph{(edge) boundary} of $\Lambda$,
meaning all edges of $G$ with exactly one vertex in $\Lambda$. 
The \emph{isoperimetric profile} $\Phi_{G}\colon (0,\infty)\to
\cb{0,1}$ of $G$ is given by 
\begin{equation}
  \label{eq:isopf}
  \Phi_{G}(t) = \inf \left\{ \frac{ |\partial\Lambda|}{|\Lambda|} :
  0<|\Lambda|\leq t, \Lambda\subset G\right\}.
\end{equation}

A graph is \emph{quasi-transitive} if the set of orbits of $V$ under
the automorphism group of $G$ is finite; see,
e.g.,~\cite[p.234]{LyonsPeres}. Transitive graphs have a single orbit,
and graphs that are transitive within each partite class have at most
two orbits. The next two lemmas summarize important geometric facts
about quasi-transitive graphs.
\begin{lemma}
  \label{lem:iso}
  Suppose $G$ is infinite, one-ended, quasi-transitive, and has
  maximum degree $\Delta$.  Then there is a constant $\Ciso>0$ such that
  \begin{equation}
    \label{eq:iso}
    \Phi_{G}(t) \geq \Ciso t^{-1/2}.
  \end{equation}
\end{lemma}
\begin{proof}
  The proof uses the notion of a \emph{rough isometry}. This
  is a map $\Psi$ from one metric space $(X,d_{X})$ to another metric
  space $(Y,d_{Y})$ for which there exists an $\alpha\geq 1$ and
  $\beta\geq 0$ such that
  \begin{enumerate}[noitemsep]
  \item For all $x,y\in X$,
    \begin{equation*}
      \alpha^{-1}d_{X}(x,y)-\beta\leq d_{Y}(\Psi(x),\Psi(y))\leq
      \alpha d_{X}(x,y) + \beta.
    \end{equation*}
  \item For all $y\in Y$, there is an $x\in X$ such that
    $d_{Y}(\Psi(x),y)\leq\beta$.
  \end{enumerate}
    
  For an introduction to rough isometries see,
  e.g.,~\cite[Section~3]{Woess}. Here we will only need the following
  facts: (i) any quasi-transitive graph $G$ is rough isometric to a
  transitive graph $G'$, (ii) any vertex transitive graph $G$ of
  polynomial volume growth is rough isometric to a Cayley graph (iii)
  for two rough isometric graphs $G$ and $G'$ there are constants
  $a,A>0$ such that the isoperimetric profiles $\Phi_{G}$ and $\Phi_{G'}$ of
  satisfy $a\Phi_{G'}(At)\leq \Phi_{G}(t)\leq A\Phi_{G'}(at)$ for all
  $t>0$, and (iv) rough isometries preserve the number of ends of a
  graph. Proofs of (i), (iii) and (iv) can be found
  in~\cite[Section~3]{Woess}, and a proof of (ii) can be found
  in~\cite[Section~7.9]{LyonsPeres}.

  By the facts above, it suffices to establish~\eqref{eq:iso} in the
  transitive setting. Let $\Lambda=(K,F)\subset G$ be a finite
  subgraph of $G$. By~\cite[Lemma~10.46]{LyonsPeres} (and the
  discussion preceding this lemma) and the assumption of bounded
  degree, there is a $c_{1}$ such that
  \begin{equation*}
    \frac{|\partial K|}{|K|} \geq \frac{c_{1}}{R(2|K|)}
  \end{equation*}
  where $R(t)$ denotes the smallest radius of a ball in $G$ that
  contains $t$ vertices. What remains is to show that our one-ended
  assumption implies there is a $c_{2}$ such that
  $R(2t)\leq c_{2}t^{1/2}$.  By the discussion
  preceding~\cite[Lemma~10.46]{LyonsPeres}, towards
  proving~\eqref{eq:iso} we may assume $G$ has at most polynomial
  volume growth. The conclusion that $R(2t)\leq c_{2}t^{1/2}$ now
  follows from the preceding paragraph, as the desired inequality is
  true for Cayley graphs of polynomial volume growth -- this is the
  content of~\cite[Theorem~7.18]{LyonsPeres}, as our assumption
  of one-endedness rules out being almost isomorphic to $\Z$ in the
  alternative presented by this theorem.
\end{proof}

\begin{lemma}
  \label{lem:vgrowth}
  Suppose $G$ is infinite, quasi-transitive, and has at most
  polynomial volume growth. Then there exist $c,C>0$ and  $d\in \N$
  such that $cn^{d}\leq |B_{n}(v)|\leq Cn^{d}$, and hence
  \begin{equation}
    \label{eq:nonexpball}
    \lim_{n\to\infty} \frac{|\partial B_{n}(v)|}{|B_{n}(v)|} = 0.
  \end{equation}
\end{lemma}
\begin{proof}
  From the proof of Lemma~\ref{lem:iso}, it suffices to consider
  Cayley graphs of at most polynomial volume growth. In this setting
  the existence of $c,C,d$ is well-known, see, e.g.,
  \cite[Section~1.1]{EasoHutchcroft2023}.
\end{proof}

\subsection{Cycle spaces, cycle bases, and basis connectivity}
\label{sec:cycle-space-cycle}

Let $G = (V,E)$ be a finite or infinite graph.  The \emph{edge space
  $\cE (G)$} of $G$ is the vector space $\mathbb{Z}_2^{E}$. The
\emph{cycle space $\cE_\cC (G) \subset \cE (G)$} is the subspace of
$\cE (G)$ spanned by the indicator vectors of cycles of $G$.  Let
$\cB$ be a collection of cycles in $G$ whose corresponding vectors
span $\cE_\cC (G)$.  We call $\cB$ a \emph{cycle basis}. Despite the
terminology, there is no linear independence condition on $\cB$ (but
`spanning set for the cycle space' is unwieldy).  A cycle basis $\cB$
is \emph{$D$-bounded} if for any edge $e \in E(G)$, the number of
edges $e'\neq e$ that are in a common basis cycle with $e$ is at most
$D$, uniformly over $e$.  Formally, a cycle basis $\cB$ is $D$-bounded
if
\begin{equation}
  \label{eq:bcb}
  \sup_{e\in E} \big| \set{e'\neq e\mid \text{there exists $B\in\cB$ such
  that $e,e'\in E(B)$}}\big| \le  D \,.
\end{equation}

A subset $E'\subset E$ of edges is \emph{basis connected (with respect
  to a cycle basis $\cB$)} if for every non-trivial bipartition
$E' = E_{1}\sqcup E_{2}$ of $E'$ there is a cycle $C\in \cB$ such that
$C\cap E_{1}$ and $C\cap E_{2}$ are both non-empty. We have
slightly abused notation by writing $C$ in place of $E(C)$ above; when
there is no risk of confusion we will do this in what
follows. Unless it is important to distinguish a particular cycle
basis $\cB$, we will just write `basis connected' in what follows. 

Let $G$ be an infinite one-ended graph. For a finite set
$H \subseteq V(G)$, let $\binf H\subset\partial H$ be all edges that
have one vertex in $H$ and one vertex in the unique infinite
component of $G \setminus H$. We will need the following facts about
$\binf H$.

\begin{lemma}\label{lem:bdry_cycle_even}
  Let $G$ be an infinite one-ended graph and let $H \subseteq V(G)$ be
  a finite set. Every cycle $C$ in $G$ must contain an even number of
  edges of $\binf H$.
\end{lemma}
\begin{proof}
  Let $A$ be the subset of vertices of $G$ in the unique infinite
  component of $G\setminus H$, and let $B=V\setminus A$. Observe that
  $\binf H$ is exactly the set of edges between $A$ and $B$. The lemma
  follows, as $A$ and $B$ partition $V(G)$, and every cycle must cross
  between the two sets an even number of times.
\end{proof}

The following is a key property relating basis connectivity and
one-endedness. It is essentially a special case
of~\cite[Lemma~2]{timar2013boundary}; we give the short proof for the
convenience of the reader. Recall the notion of $H\subset V$ being
connected from Section~\ref{sec:graphterm}.

\begin{lemma}\label{lem:bdry_basis_conn}
  Let $G$ be an infinite one-ended graph with a cycle basis $\cB$. For
  any finite connected set $H \subset V(G)$, $\binf H$ is
  basis-connected.
\end{lemma}
\begin{proof}
  If $|\binf H|=1$ there is nothing to show, so assume
  $|\binf H|\geq 2$. Let $E_1 \sqcup E_2$ be a non-trivial bipartition
  of $\binf H$.  Choose $e_1 \in E_1$ and $e_2 \in E_2$,
  $e_{i}=\{v_{i},u_{i}\}$ with $v_{i}\in H$ and $u_{i}\notin H$.

  First we construct a cycle $K$ containing $e_1$, $e_2$, and no other
  edges of $\binf H$.  Because $H$ is connected, we can find a simple
  path from $v_{1}$ to $v_{2}$ within $H$.  Because $G$ is one-ended,
  we can also find a simple path from $u_{1}$ to $u_{2}$ in the
  infinite component of $G \setminus H$. Concatenating these paths
  together with $e_1$ and $e_2$ forms a cycle.

  Decompose $K$ as a sum of cycles $C_i \in \cB$, $K = \sum_i C_i$. By
  Lemma~\ref{lem:bdry_cycle_even}, each of these cycles $C_i$
  intersects $\binf H$ an even number of times. Since $K$ has exactly
  one edge in $E_1$ and exactly one edge in $E_2$, it follows that at
  least one of the cycles $C_i$ (say $C_{1}$) must intersect
  $\binf E_1$ an odd number of times. Because the total number of
  intersections of $C_{1}$ with $\binf H$ must be even, the cycle
  $C_{1}$ must also intersect $E_2$ an odd number of times.  Thus
  $C_{1}$ has non-empty intersection with both $E_{1}$ and $E_{2}$,
  and we conclude that $\binf H$ is basis-connected.
\end{proof}

If $\pi$ is an automorphism of $G$ and
$\cB$ is a cycle basis, let $\pi\cB$ be the set of images of cycles of
$\cB$ under $\pi$. We say $\cB$ is
\emph{$\pi$-invariant} if $\pi\cB = \cB$. If $\cB$ is invariant
under all automorphisms of $G$, then we say $\cB$ is
\emph{automorphism invariant}. The next lemma allows us to restrict
attention to automorphism invariant cycle bases.
\begin{lemma}
  \label{lem:BAI}
  Suppose $G$ has maximum degree $\Delta$. If $\cB$ is a bounded cycle
  basis for $G$, then there exists a bounded and automorphism invariant
  cycle basis $\cB'$ of $G$.
\end{lemma}
\begin{proof}
  Let $L$ be the length of the longest cycles in $\cB$. If $\cB$ is
  $D$-bounded, then $L\leq D+1$. Let $\cB'$ be the set of all cycles
  of length at most $L$, and note $\cB'$ is automorphism invariant.
  Moreover, $\cB'$ is a bounded cycle basis: since $G$ has maximum
  degree $\Delta$, every edge is in a common basis cycle with at most
  $D' \leq (\Delta-1)^{L-1} \leq (\Delta-1)^D$ other edges.
\end{proof}

\subsection{Boundary Conditions for the Hard-Core Model}
\label{sec:md}

Let $\lam\colon V \to [0,\infty)$ be a vector of
\emph{activities}. This paper primarily considers \emph{bipartite
  activities}, i.e., $\lam(v)=\lam_{\even}$ if $v\in V_{\even}$ and
$\lam(v)=\lam_{\odd}$ if $v\in V_{\odd}$. In the homogeneous setting
$\lam_{\even}=\lam_{\odd}$ we simply write $\lam$.

\subsubsection{Finite Graphs and Boundary Conditions}

The \emph{hard-core model} on a finite graph $G=(V,E)$ is the
distribution $\mu_{G,\lam}$ on $\{0,1\}^{V}$ given by
\begin{equation*}
\mu_{G,\lambda}(\omega) = \mu(\omega) =
\frac{\lambda^{\omega}}{Z_{G,\lambda}} 1_{\omega\in \cI(G)}, \qquad
Z_{G, \lambda} = \sum_{\omega\in \{0,1\}^{V}} \lambda^{\omega}1_{\omega\in \cI(G)},
\end{equation*}
where $\cI(G)$ is the set of independent sets on $G$,
$\lam^{\omega} = \prod_{i\in V}\lam_{i}^{\omega_{i}}$, and
$Z_{G,\lam}$ is the \emph{partition function}. This definition agrees
with~\eqref{eq:intro} by identifying $\omega\in\{0,1\}^{V}$ with the
set of vertices for which $\omega(v)=1$; these vertices are
\emph{occupied} while the others are \emph{unoccupied}.

\emph{Boundary conditions} arise by imposing that certain
  vertices are occupied. Given $\bar{\omega}\in \{0,1\}^{U}$, $U\subset V$,
  the hard-core model with boundary conditions $\bar{\omega}$ is the
  distribution on $\{0,1\}^{V}$ given by
\begin{equation}
  \label{eq:hcbc}
  \mu^{\bar{\omega}}_{G,\lambda}(\omega) = \mu^{\bar{\omega}}(\omega) =
  \frac{\lambda^{\omega}}{Z^{\bar{\omega}}_{G,\lambda}} 1_{\omega\in
    \cI(G)}\prod_{v\in U}1_{\omega(v)=\bar{\omega}(v)},
  \qquad
  Z_{G, \lambda}^{\bar{\omega}} = \mathop{\sum_{\omega \in \{0,1\}^{V}}}_{\forall v\in U, \omega(v)=\bar{\omega}(v)}
  \lambda^{\omega}1_{\omega  \in \cI(G)}.
\end{equation}
Thus every vertex occupied in $\bar{\omega}$ remains occupied and every
vertex unoccupied in $\bar{\omega}$ remains unoccupied; the randomness of
$\mu^{\bar{\omega}}$ only concerns vertices in $V\setminus U$. 

When $G$ is bipartite, there is an important partial order $\lessdot$
on $\{0,1\}$ given by setting $\omega\lessdot\bar{\omega}$ if
$\omega(v)\leq \bar{\omega}(v)$ for $v\in V_{\even}$ and
$\omega(v)\geq \bar{\omega}(v)$ for $v\in V_{\odd}$. There are a unique
minimal and maximal elements under $\lessdot$, namely $1_{V_{\odd}}$
and $1_{V_{\even}}$. The following lemma is a well-known consequence of
the FKG inequality, see, e.g., \cite[Lemma~3.1]{van1994percolation}.
\begin{lemma}
  \label{lem:FKG-order}
  If $\omega\lessdot \bar{\omega}$, then $\mu^{\omega}$ is stochastically
  dominated by $\mu^{\bar{\omega}}$.
\end{lemma}

\subsubsection{Infinite Graphs}
\label{sec:md-infinite}
On infinite graphs the hard-core model is defined by the
Dobrushin--Lanford--Ruelle (DLR) approach, for textbook treatments
see~\cite[Chapter~6]{friedli2017statistical},
\cite{georgii2011gibbs}. In words, a measure $\mu$ on $\cI(G)$
is an \emph{infinite-volume Gibbs measure at activity $\lambda$} if
conditioned on the independent set outside of a finite set $\Lambda$
being $I$, the conditional distribution on independent sets inside
$\Lambda$ is the Gibbs measure on $\Lambda$ with the boundary
condition imposed by $I$. 

For a more precise description in the context of the hard-core model,
see~\cite{van1994percolation}. This reference also contains the
following facts that will be useful.  For infinite bipartite graphs
$G$, suppose $\Lambda_{n}\uparrow G$, i.e., $\Lambda_{n}$ is an
increasing sequence of subsets of $V$ with
$\cup_{n}\Lambda_{n}=V$. Let $\tilde \mu^{\even}_{\Lambda_{n}}$ denote
the hard-core measure on $\Lambda_{n}$ with all-even boundary
conditions, meaning all even vertices in
$U=\Lambda_{n}\cap V(\partial\Lambda_{n})$ occupied, and no odd
vertices occupied. Define $\tilde\mu^{\odd}_{\Lambda_{n}}$
analogously.  Then $\tilde\mu^{\even}_{\Lambda_{n}}$ and
$\tilde\mu^{\odd}_{\Lambda_{n}}$ converge to limiting infinite-volume
Gibbs measures $\tilde\mu^{\even}$ and $\tilde\mu^{\odd}$, and these
limits are independent of the sequence $\Lambda_{n}$. There are
multiple infinite-volume Gibbs measures if and only if
$\tilde\mu^{\even}\neq \tilde\mu^{\odd}$. Moreover,
$\tilde\mu^{\even}=\tilde\mu^{\odd}$ if and only if the single-vertex
marginals of these measures agree for all
$v\in V$~\cite[Theorem~4.18]{georgii2001random}.

\section{Pirogov--Sinai Theory: Combinatorial Steps}
\label{sec:hcm}

We now introduce the combinatorial definition of a \emph{contour} that
underpins our generalization of Pirogov--Sinai theory for the
hard-core model. This definition is adapted to an underlying infinite
graph $G$, and we impose the following condition throughout this
section.
\begin{assumption}
  \label{as:1}
  $G$ is infinite, connected, bipartite, and one-ended, and $\cB$ is
  a cycle basis of~$G$.
\end{assumption}

We will denote by $\Lambda$ a finite induced subgraph of $G$ with
$\partial \Lambda = \binf \Lambda$.  In
Section~\ref{sec:boundary-conditions} we introduce boundary conditions
for the hard-core model that are adapted to the cycle basis
$\cB$. Contours are defined in Section~\ref{sec:contours}, and
fundamental properties of contours are developed in
Section~\ref{sec:order}. Section~\ref{sec:bij} explains how to
represent the hard-core model in terms of contours. Subsequent
sections then contain preparation for the analytic aspects of
Pirogov--Sinai theory. Section~\ref{sec:peierls} defines weights and
then gives identities and estimates that lead to Peierls estimates
based on appropriate symmetry
assumptions. Section~\ref{sec:contour-polym-repr} then explains how
our definition of a contour allows for the hard-core model to be
reformulated as a polymer model, which is the starting point for the
analytic part of Pirogov--Sinai theory.

\subsection{Even and Odd Boundary Conditions}
\label{sec:boundary-conditions}

Recall that for $\Lambda$ a finite subgraph of $G$, $\partial \Lambda$
is the set of edges of $G$ with exactly one endpoint in $\Lambda$. We
say a basis cycle $C \in \cB$ \emph{exits $\Lambda$} if it contains an
edge of $\partial \Lambda$, that is, if both $V(C) \cap V(\Lambda)$
and $V(C) \cap (V(G) \setminus V(\Lambda))$ are nonempty.  We will
impose boundary conditions on (i) $\partial\Lambda$ and (ii) the set
of vertices in $\Lambda$ that are contained in a basis cycle that
exits $\Lambda$. Formally,
\begin{equation}
  \label{eq:cbbc}
  U=
  \set{v\in V(\Lambda) \mid \text{$v\in \partial\Lambda$ or there exists
    $C\in \cB$ such that $v\in V(C)$ and $C$ exits $\Lambda$}}.
\end{equation}

\emph{Odd (cycle basis) boundary conditions} on $\Lambda$ require all odd vertices
of $U$ to be occupied and all even vertices to be unoccupied;
\emph{even (cycle basis) boundary conditions} are analogous. We
  will simply say odd (and even) boundary conditions in what follows
  if no confusion will arise.
Denote the hard-core distribution and partition function with
respect to odd boundary conditions by
$\mu^{\odd}_{\Lambda, \lambda} = \mu^{\odd}_{\Lambda}$ and
$Z^{\odd}_{\Lambda, \lambda} = Z^{\odd}_\Lambda$, and
$\mu^{\even}_{\Lambda, \lambda} = \mu^{\even}_{\Lambda}$ and
$Z^{\even}_{\Lambda, \lambda} = Z^{\even}_\Lambda$ for even boundary
conditions. We write $\cI^{\even}(\Lambda)$ and
$\cI^{\odd}(\Lambda)$ for the sets of independent sets of
$\Lam$ compatible with the corresponding boundary conditions.

We can relate the boundary conditions just defined to the more
conventional boundary conditions and distributions $\tilde\mu^{\even}$
defined in Section~\ref{sec:md-infinite}. The same result evidently
holds for odd boundary conditions.
\begin{lemma}
  \label{lem:bcequiv}
  Suppose $\Lambda_{n}\uparrow G$. If $\cB$ is a bounded cycle basis,
  then the limiting measures
  $\lim_{n\to\infty}\mu^{\even}_{\Lambda_{n}}$ and
  $\lim_{n\to\infty}\tilde\mu^{\even}_{\Lambda_{n}}$ coincide.
\end{lemma}
\begin{proof}
  The vertices occupied by the cycle basis boundary conditions on
  $\Lambda_{n}$ in the definition of $\mu^{\even}_{\Lambda_{n}}$ are a
  superset of those occupied by the standard boundary conditions in
  the definition of $\tilde \mu^{\even}_{\Lambda_{n}}$, so $\tilde
  \mu^{\even}_{\Lambda_{n}}$ is stochastically dominated by
  $\mu^{\even}_{\Lambda_{n}}$ by Lemma~\ref{lem:FKG-order}. On the
  other hand, for some finite $m$, the boundary condition defining
  $\mu^{\even}_{\Lambda_{n+m}}$ is a subset of that defining $\tilde
  \mu^{\even}_{\Lambda_{n}}$ since $G$ has a bounded cycle basis and
  $\Lambda_{n}\uparrow G$. This implies $\mu^{\even}_{\Lambda_{n+m}}$
  is stochastically dominated by $\tilde
  \mu^{\even}_{\Lambda_{n}}$. The claim follows since
  $\lim_{n\to\infty}\tilde \mu^{\even}_{\Lambda_{n}}$ exists as was
  recalled in Section~\ref{sec:md-infinite}. 
\end{proof}

\subsection{Contours and Compatibility}
\label{sec:contours}

Recall that for a collection of edges $E' \subseteq E(G)$,
$G \setminus E'$ denotes the graph $G'$ with vertex set $V(G)$ and
edge set $E(G) \setminus E'$. A \emph{contour} $\gamma$ of $G$ is a
finite nonempty basis-connected subset of $E(G)$ such that for each
connected component $A$ of $G \setminus \gamma$, the vertices of $A$
incident to edges of $\gamma$ have the same parity. See
Figure~\ref{fig:contour}. This definition implies that
$G \setminus \gamma$ has at least two connected components. Let
$\cC(G)$ be the set of all contours of $G$.

\begin{figure}
  \centering
  \begin{tikzpicture}
    \draw[black,dotted] (-.5,-.5) grid (2.5,2.5);

    \draw[eedge] (1,1) -- (1,2);
    \draw[eedge] (1,1) -- (1,0);
    \draw[eedge] (1,1) -- (0,1);
    \draw[eedge] (1,1) -- (2,1);
    
    \node[evE] at (0,0) {};
    \node[evO] at (1,0) {};
    \node[evE] at (2,0) {};
    \node[evO] at (0,1) {};
    \node[evE] at (1,1) {};
    \node[evO] at (2,1) {};
    \node[evE] at (0,2) {};
    \node[evO] at (1,2) {};
    \node[evE] at (2,2) {};
  \end{tikzpicture}
  \caption{The smallest contour (wavy lines) on $\Z^2$ when $\cB$
    consists of the length four cycles
    $(x,x+e_1,x+e_{1}+e_{2},x+e_{2},x)$ for $x\in \Z^{2}$,
    $e_{1},e_{2}$ the standard unit basis vectors in $\R^{2}$.}
  \label{fig:contour}
\end{figure}
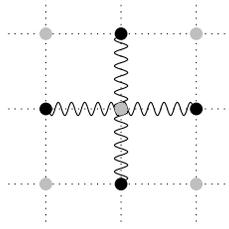

\begin{lemma}
  \label{lem:diff_comps}
  The edges of a contour $\gamma$ have their endpoints in different
  components of $G \setminus \gamma$.
\end{lemma}
\begin{proof}
  For any edge $e$ one endpoint is odd and one endpoint is even.  
\end{proof}

Given an independent set $I$ we say an edge $e$ is \emph{unoccupied
  (by $I$)} if $e\cap I = \emptyset$. The next proposition will not be
needed for our subsequent developments, but it provides valuable
intuition for the meaning of contours, and similar constructions will
be used in the sequel.
\begin{prop}
  \label{prop:single_contour_config}
  For each $\gamma\in \cC(G)$ there is an independent set
  $I\in \cI(G)$ whose unoccupied edges are exactly the edges
  of $\gamma$.
\end{prop}
\begin{proof}
  For each $\gamma$ we will construct an independent set $I$ whose
  unoccupied edges are exactly those of $\gamma$. For each connected
  component $H$ of $G\setminus\gamma$, all vertices in $H$ that are
  endpoints of edges in $\gamma$ have the same parity. Let
  $\mathcal{H}_{\odd}$ denote the union of the components for which
  these vertices are even, and $\mathcal{H}_{\even}$ the union of
  components for which they are odd. Let $I$ be the independent set
  consisting of all odd vertices in $\mathcal{H}_{\odd}$ and all even
  vertices in $\mathcal{H}_{\even}$. Now observe that every edge $e$
  of $\gamma$ is unoccupied by $I$.  Moreover, there are no unoccupied
  edges that are not in $\gamma$, as every edge interior to a
  component of $\mathcal{H}_{\even}$ or $\mathcal{H}_{\odd}$ contains
  a vertex of each parity.
\end{proof}

Two contours $\gamma$ and $\gamma'$ are \emph{compatible} if their
union is not basis connected.  Otherwise the contours are declared
\emph{incompatible}.
\begin{lemma}
  \label{lem:comp-disjoint}
  If $\gamma$ and $\gamma'$ are compatible, then they are disjoint.
\end{lemma}
\begin{proof}
  The existence of a common edge in $\gamma$ and $\gamma'$ implies
  that any non-trivial bipartition of $\gamma\cup \gamma'$ induces a
  non-trivial bipartition of at least one of $\gamma$ and $\gamma'$.
\end{proof}
A set $\Gamma$ of contours is \emph{compatible} if all contours in
$\Gamma$ are pairwise compatible. The following lemma is almost
immediate.

\begin{lemma}
  \label{lem:comp-maximal}
  Suppose $\Gamma$ is a collection of compatible contours. Then
  $\Gamma$ is the set of maximal basis connected
  subsets of $\bigcup_{\gamma\in\Gamma}\gamma\subset E(G)$.
\end{lemma}
\begin{proof}
  Let $E(\Gamma)=\bigcup_{\gamma\in\Gamma}\gamma$. By compatibility, a
  maximal basis connected subset of $E(\Gamma)$ cannot be a superset
  of a contour $\gamma$. A maximal basis connected subset of
  $E(\Gamma)$ cannot be non-empty proper subset of a contour $\gamma$
  as this would contradict $\gamma$ being a contour.
\end{proof}

In Section~\ref{sec:bij} we will show that given an independent set
$I$, we obtain a set of compatible contours by considering the set of
maximal basis-connected components of the unoccupied edges of $I$. Not
all sets of compatible contours can arise in this way,
however. Section~\ref{sec:order} develops some preliminaries that will
play a role in our description of the sets of contours that can arise
in Section~\ref{sec:bij}.

\subsection{Ordering contours}
\label{sec:order}

By Assumption~\ref{as:1}, there is a unique infinite component of
$G \setminus \gamma$, and we call it the \emph{exterior component} of
$G \setminus \gamma$.  All other components of $G \setminus \gamma$
are called \emph{interior components}. More generally these
  notions makes sense when removing any finite set of edges, e.g., a
  collection of contours.  We write $\inte\gamma$ for the set of
vertices in interior components.  This section verifies that these
notions of interior and exterior behave as one would intuitively
expect.

If $\gamma$ and $\gamma'$ are compatible contours, say
$\gamma' \prec\gamma$ if the endpoints of edges in $\gamma'$ are all
contained in interior components of $G\setminus\gamma$.  We read
$\gamma'\prec \gamma$ as `$\gamma'$ is contained in $\gamma$'. We will
write that $\gamma$ is \emph{exterior to} $\gamma'$ or $\gamma'$ is
\emph{interior to} $\gamma$, depending on what is grammatically
convenient.

The following lemma is helpful for deducing properties of the relation
$\preceq$.
\begin{lemma}
  \label{lem:order}
  Suppose $\gamma$ and $\gamma'$ are compatible contours. If $\gamma'$
  is incident to a vertex in an interior component $H$ of
  $G\setminus\gamma$, then $\gamma'$ is contained in $H$.  In
  particular, $\gamma'\prec\gamma$.
\end{lemma}
\begin{proof}
  Towards a contradiction, suppose not. Then there are two vertices
  $u$ and $v$ in edges of $\gamma'$ in distinct components of
  $G \setminus \gamma$. Note that $\{u,v\}$ can't be an edge of
  $\gamma'$, as compatibility would imply $\{u,v\}\notin \gamma$,
  which would imply $u$ and $v$ are in the same component of
  $G\setminus\gamma$.  Compatibility, via
  Lemma~\ref{lem:comp-disjoint}, further implies that there are
  distinct edges $e_{u}$ and $e_{v}$ of $\gamma'$ that contain $u$ and
  $v$, respectively. Since $e_{u}$ and $e_{v}$ are basis connected,
  there is a basis cycle $C$ connecting them. Our assumption that
  $u,v$ are in distinct components of $G\setminus\gamma$ implies $C$
  contains an edge of $\gamma$. This is a contradiction, as it implies
  $\gamma$ and $\gamma'$ are not compatible.
\end{proof}
  
Let $|\inte \gamma|$ denote the number of vertices contained in the
interior components of $\gamma$.
\begin{lemma}
  \label{lem:spo_order}
  The relation $\prec$ has the following properties:
  \begin{enumerate}[nosep]
  \item It is a strict partial order on contours.
  \item If $\gamma'\prec\gamma$, then
    $\inte \gamma'\subsetneq \inte\gamma $. In particular,
      $|\inte \gamma'| < |\inte \gamma|$.
    \item If $\gamma\prec\gamma_1$,
        $\gamma\prec\gamma_2$, and
        $\{\gamma,\gamma_{1},\gamma_{2}\}$ are a
        compatible set of contours, then either
        $\gamma_{1}\prec\gamma_{2}$ or $\gamma_{2}\prec\gamma_{1}$.
  \end{enumerate}
\end{lemma}
\begin{proof}
  \textbf{Claim 1.} No contour can contain itself, so $\prec$ is an irreflexive
  relation. Lemma~\ref{lem:order} implies $\prec$ is
  transitive. Irreflexivity and transitivity imply asymmetry, so
  $\prec$ is a strict partial order.

  \textbf{Claim 2.} Since $\gamma'\prec\gamma$ means $\gamma'$
  contains an edge with an endpoint in $\inte\gamma$,
  Lemma~\ref{lem:order} implies $\gamma'$ is contained in
  $\inte\gamma$. This implies $\inte\gamma'\subset\inte\gamma$, as if
  a vertex $v\in \inte\gamma'$ was not in $\inte\gamma$, then an edge
  of $\gamma'$ would not be included in $\inte\gamma$ (by following a
  path from $v$ to infinity). The inclusion is strict since the
  endpoints of edges of $\gamma'$ include vertices not in
  $\inte\gamma'$.

  \textbf{Claim 3.} Note that $\inte\gamma_{1}$ and
  $\inte\gamma_{2}$ have a vertex in common, as $\inte\gamma$ is
  contained in each of these sets. Hence either an
  endpoint of an edge of $\gamma_{2}$ is contained in
  $\inte\gamma_{1}$ or vice versa. The conclusion follows by
  Lemma~\ref{lem:order}.
\end{proof}

Given a collection of compatible contours $\Gamma$, $\gamma\in\Gamma$
is \emph{external} if is not contained in any other contour in
$\Gamma$.  The third item of Lemma~\ref{lem:spo_order} reveals a
product structure on sets of compatible external contours. Let
$\Gamma$ be a set of contours in which each contour is external. Then
  \begin{equation}
    \label{eq:product}
    \set{ \Gamma' \mid \text{the external contours of $\Gamma'$ are
      $\Gamma$}} = \prod_{\gamma\in\Gamma} \set{\tilde{\Gamma} \mid
    \text{$\gamma$ is the unique external contour of $\tilde{\Gamma}$}}.
  \end{equation}

  Given a collection of contours $\Gamma$, write $G\setminus\Gamma$
  for the graph $G$ with the edges contained in contours in
  $\Gamma$ removed. 
  \begin{lemma}
    \label{lem:ext-connect}
    Let $\Gamma$ be a finite collection of compatible contours. Suppose
    $\gamma\in\Gamma$ is external. Let $v\in\ext\gamma$ be incident to
    an edge of $\gamma$. Then $v$ is in the exterior component of
    $G\setminus\Gamma$.
  \end{lemma}
  
  \begin{proof}
    Let $v\in\ext\gamma$ be incident to an edge of $\gamma$, and let
    $H_{v}$ denote the connected component of $v$ in
    $G\setminus\Gamma$. Suppose, towards a contradiction, that $H_{v}$
    is not the exterior component of $G\setminus\Gamma$. Then $H_{v}$
    is finite, and hence $\binf H_{v}$ is basis connected by
    Lemma~\ref{lem:bdry_basis_conn}. Thus edges contained in a subset
    $\Gamma'\subset \Gamma\setminus\{\gamma\}$ separate $v$ from
    infinity. Since $\binf H_{v}$ is basis connected, there can be at
    most one contour in $\Gamma'$. But $\Gamma'= \{\gamma'\}$ is
    a contradiction, as in this case $\gamma\prec\gamma'$ by
    Lemma~\ref{lem:order}, contradicting $\gamma$ being
    external. Hence $H_{v}$ must be the exterior component of
    $G\setminus\Gamma$. 
  \end{proof}

\subsection{Contour representations of independent sets}
\label{sec:bij}

Let $\Lam$ be a finite induced subgraph of $G$. The initial step in
carrying out Pirogov--Sinai theory is to find a representation of the
partition functions $Z^{\even}_{\Lam}$ and $Z^{\odd}_{\Lam}$ in terms
of contours.  We achieve this in Proposition~\ref{lem:bijection}
below after establishing some further terminology. 

Let $\gamma$ be a contour.  A connected component $H$ of
$G \setminus \gamma$ is called an \emph{even (occupied) component} if
every vertex in $H$ incident to an edge of $\gamma$ is odd. The
terminology refers to the fact that vertices incident to edges of
$\gamma$ are unoccupied in the construction of contours from an
independent set used in the proof of
Proposition~\ref{prop:single_contour_config}.  Similarly, a connected
component $H$ of $G \setminus \gamma$ is an \emph{odd (occupied)
  component} if every vertex in $H$ incident to an edge of $\gamma$ is
even.  Every edge of $\gamma$ has one endpoint in an even component
and one endpoint in an odd component.  See
Figure~\ref{fig:Z2-hc-contours}.

We will label contours according to their exterior components: call a
contour $\gamma$ an \emph{even contour} if the exterior component of
$G \setminus \gamma$ is even, and call $\gamma$ an \emph{odd contour}
if the exterior component is odd.  Let $\inte_{\odd} \gamma$ be all
vertices in an odd interior component of $G \setminus \gamma$, and
$\inte_{\even} \gamma$ be all vertices in an even interior component
of $\gamma$. Thus
$\inte\gamma = \inte_{\even}\gamma\cup\inte_{\odd}\gamma$. Note that
$\inte_{\odd} \gamma$ and $\inte_{\even} \gamma$ may each induce
disconnected subgraphs.  We will sometimes abuse notation and identify
$\inte_{\even}\gamma$ and $\inte_{\odd}\gamma$ with the subgraphs they
induce.

\begin{figure}
  \centering
  \begin{tikzpicture}
    \draw[black,dotted] (-.5,-.5) grid (5.5,5.5);
    \node[ovE] at (0,0) {};
    \node[evO] at (1,0) {};
    \node[ovE] at (2,0) {};
    \node[evO] at (3,0) {};
    \node[ovE] at (4,0) {};
    \node[evO] at (5,0) {};
    \node[evO] at (0,1) {};
    \node[ovE] at (1,1) {};
    \node[evO] at (2,1) {};
    \node[evE] at (3,1) {};
    \node[evO] at (4,1) {};
    \node[ovE] at (5,1) {};
    \node[ovE] at (0,2) {};
    \node[evO] at (1,2) {};
    \node[evE] at (2,2) {};
    \node[ovO] at (3,2) {};
    \node[evE] at (4,2) {};
    \node[evO] at (5,2) {};
    \node[evO] at (0,3) {};
    \node[evE] at (1,3) {};
    \node[ovO] at (2,3) {};
    \node[evE] at (3,3) {};
    \node[evO] at (4,3) {};
    \node[ovE] at (5,3) {};
    \node[ovE] at (0,4) {};
    \node[evO] at (1,4) {};
    \node[evE] at (2,4) {};
    \node[evO] at (3,4) {};
    \node[ovE] at (4,4) {};
    \node[evO] at (5,4) {};
    \node[evO] at (0,5) {};
    \node[ovE] at (1,5) {};
    \node[evO] at (2,5) {};
    \node[ovE] at (3,5) {};
    \node[evO] at (4,5) {};
    \node[ovE] at (5,5) {};    
  \end{tikzpicture}
    \hspace{4em}
    \begin{tikzpicture}
    \draw[black,dotted] (-.5,-.5) grid (5.5,5.5);

    \draw[eedge] (0,3) -- (1,3);
    \draw[eedge] (1,3) -- (1,2);
    \draw[eedge] (1,2)-- (2,2);
    \draw[eedge] (2,2)-- (2,1);
    \draw[eedge] (2,1) -- (3,1);
    \draw[eedge] (3,1)-- (3,0);
    \draw[eedge] (3,1) -- (4,1);
    \draw[eedge] (4,1) -- (4,2);
    \draw[eedge] (4,2) -- (5,2);
    \draw[eedge] (4,2) -- (4,3);
    \draw[eedge] (4,3) -- (3,3);
    \draw[eedge] (3,3)-- (3,4);
    \draw[eedge] (3,4)-- (2,4);
    \draw[eedge] (2,4) -- (2,5);
    \draw[eedge] (2,4) -- (1,4);
    \draw[eedge] (1,4) -- (1,3);
    \draw[ultra thick, gray!70] (1,3) -- (2,3) -- (2,4) -- (2,3) --
    (3,3) -- (3,2) -- (2,2) -- (2,3);
    \draw[ultra thick, gray!70] (3,1) -- (3,2) -- (4,2);

    \node[evE] at (0,0) {};
    \node[evO] at (1,0) {};
    \node[evE] at (2,0) {};
    \node[evO] at (3,0) {};
    \node[evE] at (4,0) {};
    \node[evO] at (5,0) {};
    \node[evO] at (0,1) {};
    \node[evE] at (1,1) {};
    \node[evO] at (2,1) {};
    \node[evE] at (3,1) {};
    \node[evO] at (4,1) {};
    \node[evE] at (5,1) {};
    \node[evE] at (0,2) {};
    \node[evO] at (1,2) {};
    \node[evE] at (2,2) {};
    \node[evO] at (3,2) {};
    \node[evE] at (4,2) {};
    \node[evO] at (5,2) {};
    \node[evO] at (0,3) {};
    \node[evE] at (1,3) {};
    \node[evO] at (2,3) {};
    \node[evE] at (3,3) {};
    \node[evO] at (4,3) {};
    \node[evE] at (5,3) {};
    \node[evE] at (0,4) {};
    \node[evO] at (1,4) {};
    \node[evE] at (2,4) {};
    \node[evO] at (3,4) {};
    \node[evE] at (4,4) {};
    \node[evO] at (5,4) {};
    \node[evO] at (0,5) {};
    \node[evE] at (1,5) {};
    \node[evO] at (2,5) {};
    \node[evE] at (3,5) {};
    \node[evO] at (4,5) {};
    \node[evE] at (5,5) {};  
  \end{tikzpicture}
  \caption{Left: a hard-core configuration on a subset of
    $\mathbb{Z}^{2}$. Large circles indicate vertices contained in the
    hard-core configuration. Dark and light shading indicates even and
    odd vertices, respectively. Right: the contour $\gamma$ (wavy
    lines) corresponding to the hard-core configuration on the
    right if $\cB$ is the cycle basis from Figure~\ref{fig:contour}. The solid grey edges are edges of the interior component
    $\inte\gamma$. In this example
      $\inte\gamma=\inte_{\even}\gamma$. The contour is an
    \emph{odd} contour as the vertices in the exterior component of
    $\gamma$ that are contained in edges of $\gamma$ are even.}
    \label{fig:Z2-hc-contours}
\end{figure}
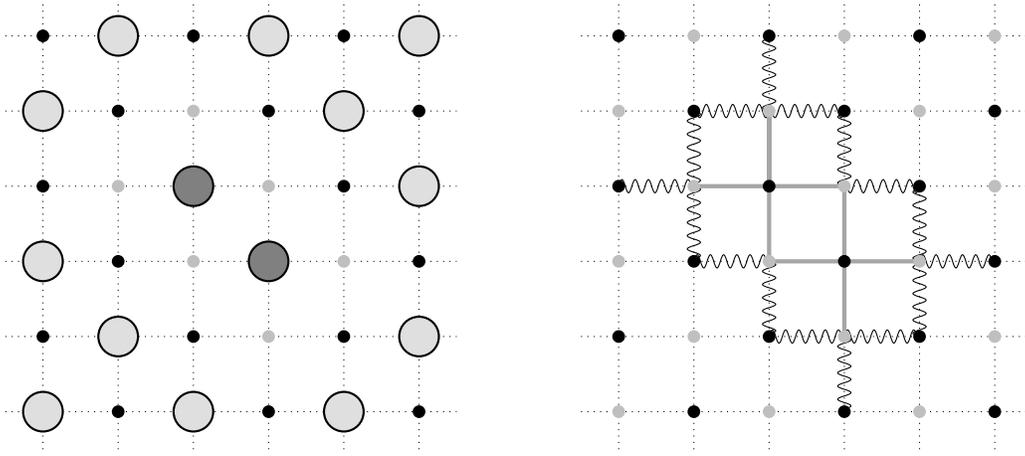

Let $\Gamma$ be a collection of compatible contours. We say $\Gamma$
is \emph{external even} if each external contour in $\Gamma$ is an
even contour. We say $\Gamma$ is \emph{matching} if for any connected
component $H$ of $G \setminus \Gamma$, the vertices of $H$ incident to
an edge of a contour in $\Gamma$ have the same parity.  Matching is
not a pairwise condition on the contours in $\Gamma$; this is
unimportant in the present section, but will require consideration in
Section~\ref{sec:contour-polym-repr} below.

Let $\cC(\Lambda)\subset\cC(G)$ denote the subset of contours
contained in $\Lambda$. Define $\cCb(\Lambda)\subset \cC(\Lambda)$ to
be the subset of contours that are disjoint from all cycles $C\in \cB$
that exit~$\Lambda$.

\begin{prop}
  \label{lem:bijection}
  Let $\Lambda$ be a finite induced subgraph of $G$ with
  $\partial\Lambda=\binf\Lambda$. There is a bijection between
  $\cI^{\even}(\Lambda)$ 
  and collections of compatible, matching, external even contours in
  $\cCb(\Lambda)$. Analogously, there is a bijection between
  $\cI^{\odd}(\Lambda)$ and collections of compatible,
  matching, external odd contours in $\cCb(\Lambda)$.
\end{prop}
\begin{proof}
  We prove the statement for even boundary conditions; the proof for
  odd boundary conditions is exactly the same. Recall that even
  boundary conditions on $\Lambda$ mean that all even vertices in
  $\partial\Lambda$ or on a basis cycle that exits $\Lambda$ are fixed
  to be occupied.  It follows that for $I\in\cI^{\even}(\Lam)$, there
  are no unoccupied edges on basis cycles that exit $\Lambda$.

  \textbf{Contours determine independent sets.} First, we show that
  every collection of compatible, matching, external even contours
  $\Gamma \subseteq \cCb(\Lambda)$ determines a distinct independent
  set with even boundary conditions. This is similar to the proof of
  Proposition~\ref{prop:single_contour_config}: we construct an
  independent set $I$ whose unoccupied edges are exactly the edges in
  contours in $\Gamma$.

  Let $H$ be a connected component of $\Lambda\setminus\Gamma$ that is
  incident to an edge in $\partial\Lambda$. We first claim that every vertex
  in $H$ incident to an edge of $\Gamma$ is odd. Since
  $\partial\Lambda=\binf\Lambda$, $H$ is contained in the exterior
  component of $G\setminus\Gamma$. 
  Hence by item 2 of Lemma~\ref{lem:spo_order}, vertices in $H$ incident to
  edges of $\Gamma$ are in fact incident to edges of external contours
  in $\Gamma$. Since all external contours are even, such vertices are
  odd.
 
  Consider the components of $\Lambda \setminus \Gamma$ that are
  incident to an edge of $\partial \Lambda$.  We begin constructing
  $I$ by setting all even vertices in these components to be
  occupied. The previous paragraph implies edges in $\Gamma$ with one
  endpoint $v$ in these components have $v$ unoccupied by $I$. Every
  edge with both endpoints in one of these components has exactly one
  vertex occupied by $I$. To verify that $I$ is compatible with even
  boundary conditions we must check that edges $e$ in $\Lambda$ on
  basis cycles that exit $\Lambda$ have their even vertices
  occupied. It suffices to observe that if $e$ is such an edge, then
  $e$ must be in one of the components under consideration: the path
  from $e$ to $\partial \Lambda$ along such a basis cycle cannot
  contain any edges of contours in $\cCb(\Lambda)$, and thus cannot
  contain edges in~$\Gamma$.
 
  We now consider the other components of $\Lambda \setminus \Gamma$.
  Let $H$ be one such component. Because $\Gamma$ is matching, all
  vertices of $H$ that are endpoints of an edge of $\Gamma$ have the
  same parity. If these vertices are all even, we add all odd vertices
  of $H$ to $I$; if they are odd, we add all even vertices of $H$ to
  $I$.  After this, every edge with both endpoints in $H$ now has
  exactly one of its endpoints occupied by $I$. For any edge $e$ with
  exactly one endpoint in $H$, meaning $e$ is an edge of a contour in
  $\Gamma$, its endpoint within $H$ remains unoccupied by
  $I$. Repeating this process for all other connected components of
  $\Lambda \setminus \Gamma$ results in an independent set with even
  boundary conditions, and the set of unoccupied edges is exactly the
  set of edges that are in a contour in $\Gamma$. By
  Lemma~\ref{lem:comp-maximal}, this last property
  implies the map we have described is injective. In the next part of the
  proof we will construct the inverse.

  \textbf{Independent sets determine contours.} Next, we show that
  every independent set $I\in \cI^{\even}(\Lambda)$
  determines a distinct collection of compatible, matching, external
  even contours from $\cCb(\Lambda)$. The contours will be comprised
  exactly of the (possibly empty) set $S$ of edges of $G$ unoccupied by $I$.

  Consider a connected component $H$ of $G\setminus S$.  Let
  $a,b\in V(H)$ be two vertices of $H$ that are incident to edges of
  $S$. By construction, $a$ and $b$ are unoccupied in $I$.  Because
  $H$ is connected, there exists a path from $a$ to $b$ in $H$ that
  does not use any edge in $S$. The parity of vertices in this path
  alternate, as does the status of each vertex as being
  occupied/unoccupied. Hence $a$ and $b$ have the same parity. This
  implies that for each connected component $H$ of $G\setminus S$ the
  vertices incident to $S$ have the same parity.
  
  We will divide $S$ into contours by declaring the maximal basis
  connected subsets of $S$ to be contours.  To be sure this is
  well-defined, we must verify that if $\gamma$ is a maximal set of
  basis-connected edges from $S$ then $\gamma$ is in fact a valid
  contour, i.e., that in any component of $G\setminus\gamma$ all
  vertices incident to edges of $\gamma$ have the same
  parity. Suppose, for the sake of contradiction, that there is a
  connected component $H$ of $G\setminus\gamma$ that contains vertices
  $u$ and $v$, both incident to edges of $\gamma$, but of different
  parities.  Let $H_u$ be the connected component of $u$ in
  $G\setminus S$.  The previous paragraph implies $v \notin H_u$, as
  $v$ is incident to an edge of $\gamma\subset S$ and $v$ has the
  opposite parity of $u$. By Lemma~\ref{lem:bdry_basis_conn},
  $\binf H_u$ is basis connected. This implies $\binf H_u$ is
  contained in $\gamma$, as $\binf H_{u}$ consists of unoccupied edges
  that are basis connected to an edge of $\gamma$ incident to $u$, and
  because $\gamma$ is a maximal basis connected subset of
    unoccupied edges.  Since $v \not \in H_u$, any path from $u$ to
  $v$ in $H$ must pass through an edge of $\binf H_u$, and hence
  through an edge of $\gamma$. This contradicts the initial assumption
  that $u$ and $v$ are in the same connected component of
  $G\setminus\gamma$.  We conclude that in any component of
  $G\setminus\gamma$ all vertices incident to edges of $\gamma$ have
  the same parity, and therefore $\gamma$ is a valid contour.
  
  Let $\Gamma$ be the set of contours obtained by splitting $S$ into
  maximal basis-connected subsets.  Because $I$ has even boundary
  conditions, no edges in a basis cycle that exits $\Gamma$ are
  unoccupied, so all of these contours are in $\cCb(\Lambda)$.  Since
  $\partial\Lambda=\binf\Lambda$, $\Gamma$ is matching because the
  union of these contours is $S$ (by the parity-occupation argument
  used in the second paragraph of this part of the proof). Using this
  argument once more shows the external contours of $\Gamma$ are even
  by Lemma~\ref{lem:ext-connect} and the hypothesis
  $\partial\Lambda=\binf\Lambda$ since $I$ has even boundary
  conditions.  The contours in $\Gamma$ are pairwise compatible since
  $S$ was split into maximal basis-connected subsets.  Thus $\Gamma$
  is a set of compatible, matching, external even contours from
  $\cCb(\Lambda)$, comprised exactly of the unoccupied edges in $I$.

  To conclude, we must show this map is injective, i.e., that the set
  $S$ of unoccupied edges determines the independent set. The
    unoccupied edges do determine an independent set of $G$. Since
    $\Lambda$ is an induced subgraph of $G$, the independent sets in
    $\Lambda$ are a subset of those in $G$, implying
    injectivity.
\end{proof}

The next lemma says that the bijection of
Proposition~\ref{lem:bijection} is `local', in the sense that interior
components of contours satisfy the hypothesis of the proposition.
\begin{lemma}
  \label{lem:int_ind}
  For any finite connected $\Lambda$ that arises as a component of
  $\inte\gamma$ for some contour $\gamma$, $\Lambda$ is an induced
  subgraph of $G$ and $\partial\Lambda = \binf\Lambda$.
\end{lemma}
\begin{proof}
  By definition, the vertices in $\Lam$ split into vertices $V_{1}$
  incident to some edge of $\partial\Lambda$, and vertices $V_{2}$
  that are only incident to vertices of $\Lam$. Since vertices in
  $V_{1}$ have the same parity, there cannot be any edges in $G$
  between vertices in $V_{1}$. Hence the set of edges induced by $V_{1}$
  and $V_{2}$ is the same as the set of edges containing an endpoint
  in $V_{2}$, i.e., $\Lam$ is an induced subgraph.  The second
  claim is immediate as $G$ is one-ended.
\end{proof}

The next lemma characterizes subgraphs of $G$ that can arise as
interior components of contours. This determines the scope of the set
of finite graphs to which our algorithmic results
(Theorem~\ref{thm:expansion}) apply.
\begin{lemma}
  \label{lem:allowed}
  Suppose $\cB$ is bounded. A finite connected subgraph
  $\Lambda$ of $G$ is a connected component of $\inte_{\even} \gamma$
  for some contour $\gamma$ if and only if all vertices $v\in \Lambda$
  contained in $\partial\Lambda$ are odd.
\end{lemma}
\begin{proof}
  We first construct a contour for a given $\Lambda$. Consider the
  independent set $I$ whose vertex set is the union of (i) even
  vertices in $\Lambda$ and (ii) odd vertices in $\Lambda^{c}$. This
  is an independent set as no even vertex in $\Lambda$ is connected to
  a vertex outside $\Lambda$. It is an independent set with odd
  boundary conditions by considering $\Lambda\subset\Lambda'$ for a
  sufficiently large $\Lambda$, see the proof of
  Lemma~\ref{lem:bcequiv}; it is here that we use $\cB$ bounded.
  Moreover, the set of unoccupied edges of $I$ is exactly
  $\partial\Lambda$ by construction. Proposition~\ref{lem:bijection}
  implies $\partial\Lambda$ is a contour $\gamma$, and by construction
  $\Lambda$ is a connected component of $\inte_{\even} \gamma$.

  For the converse, vertices of $\inte_{\even}\gamma$ incident to
  $\gamma$ are odd by definition. 
\end{proof}

\subsection{Contour weights}
\label{sec:peierls}

Proposition~\ref{lem:bijection} related independent sets and
collections of contours. In this section we re-express the weight of
an independent set in terms of a weight function on contours.  Recall
from Section~\ref{sec:md} that the weight of an independent set $I$ is
$\lambda^{I}$. Explicitly, in the bivariate setting (which is all that
is considered in this section), this simplifies to
$\lam^{I} = \lam_{\even}^{|I\cap V_\even|} \lam_{\odd}^{|I\cap
  V_\odd|}$.

For a contour~$\gamma$, define
\begin{equation}
  \label{eq:bbipartite}
  b_{\even}(\gamma) =
  \begin{cases}
    \phantom{-}| \inte_\odd \gamma \cap V_\even| & \text{$\gamma$ even} \\
    -|\inte_\even \gamma \cap V_\even| & \text{$\gamma$ odd}
  \end{cases},
  \qquad
  b_{\odd}(\gamma) =
  \begin{cases}
    -| \inte_\odd \gamma \cap V_\odd| & \text{$\gamma$ even} \\
    \phantom{-}|\inte_\even \gamma \cap V_\odd| & \text{$\gamma$ odd}
  \end{cases}.
\end{equation}
The next lemma says that $b_{\even}(\gamma)$ and $b_{\odd}(\gamma)$
measure the change in the number of even and odd occupied vertices due
to the presence of contour $\gamma$.

\begin{lemma}
  \label{lem:b_defn}
  Let $I$ and $I'$ be two independent sets in $\cI^{\even}$
  corresponding to compatible, matching, external even collections of
  contours $\Gamma$ and $\Gamma'$, respectively.  Suppose
  $\gamma \notin \Gamma'$, and $\Gamma = \Gamma' \cup \gamma$. Then
  $|I'\cap V_\even| - |I\cap V_\even| = b_{\even}(\gamma)$ and
  $|I'\cap V_\odd| - |I\cap V_\odd| = b_{\odd}(\gamma)$.
\end{lemma}
\begin{proof}
  Suppose $\gamma$ is an even contour. Consider $\inte_\odd\gamma$,
  and note there cannot be other contours $\gamma'\in \Gamma'$
  interior to $\gamma$ contained in these odd components, as in that
  case it would not be possible for both $\Gamma$ and $\Gamma'$ to be
  matching: the vertices incident to such a contour $\gamma'$
  contained in the exterior component of $G\setminus\gamma'$ would be
  even, contradicting that they must be odd since $\gamma$ is an even
  contour. Thus, recalling Proposition~\ref{lem:bijection}, it must be
  that all odd vertices of these components are occupied in $I$, and
  all even vertices of these components are occupied in $I'$. On the
  other hand, $I = I'$ on even components of $G\setminus\gamma$.  It
  follows that
  \begin{equation*}
    |I'\cap V_\even| - |\inte_\odd \gamma \cap V_\even| = |I\cap V_\even|,
    \qquad |I'\cap V_\odd| + |\inte_\odd \gamma \cap V_\odd| = |I\cap V_\odd|,
  \end{equation*}
  which (after re-arrangement) is the desired conclusion.  The case
  when $\gamma$ is odd is analogous.
\end{proof}

We define the \emph{weight} of contour $\gamma$ to be 
\begin{equation}
  w_\gamma =
  \lam_{\even}^{-b_{\even}(\gamma)}\lam_{\odd}^{-b_{\odd}(\gamma)} = 
  \begin{cases}
    \lam_{\even}^{-|\inte_\odd \gamma \cap V_\even|} \lam_{\odd}^{|\inte_\odd\gamma \cap V_\odd|}  & \gamma \text{ even,} \\
     \lam_{\even}^{|\inte_\even\gamma \cap V_\even |}\lam_{\odd}^{-|\inte_\even\gamma \cap V_\odd| } & \gamma \text{ odd.}
  \end{cases} 
\end{equation}

\begin{lemma}
  \label{lem:contour_weights}
  Suppose $\Lambda$ is a finite induced subgraph of $G$.  For
  $I \in \cI^{\even}(\Lambda)$, let $\Gamma$ be the corresponding
  collection of compatible, matching, external even contours given by
  the bijection of Proposition~\ref{lem:bijection}.  Then
  \begin{equation*}
    \lambda^{I} = \lam_{\even}^{|V_\even \cap \Lambda|} \prod_{\gamma
      \in \Gamma} w_\gamma. 
  \end{equation*}
The same holds with odd replacing even. 
\end{lemma}
\begin{proof}
  Consider even boundary conditions. For $\Gamma_0 = \emptyset$ the
  corresponding independent set $I_{0}$ has no unoccupied edges, i.e.,
  it is $I_0 = V_\even \cap \Lambda$. This independent set has weight
  $\lam^{I_{0}}=\lam_{\even}^{|V_\even \cap \Lambda|}$. An analogous
  argument applies for $\Gamma_{0}=\emptyset$ with odd boundary
  conditions.

  Let $\gamma_1,\gamma_{2},\dots$ be an ordering of the contours of
  $\Gamma$ such that $j > i$ if $\gamma_j \prec \gamma_i$ for all
  $i\neq j$.  Let $\Gamma_i = \cup_{j = 1}^i \{\gamma_j\}$. The chosen
  order on contours ensures that each $\Gamma_i$ is a collection of
  compatible, matched, external even contours, and by
  Lemma~\ref{lem:bijection} each corresponds to an independent set
  $I_i$ with even boundary conditions.  Recalling the definition of
  contours weights, the lemma now follows by induction on $i$, as
  Lemma~\ref{lem:b_defn} implies
  $ \lam^{I_i} / \lam^{I_{i-1}} =
  \lam_{\even}^{-b_{\even}(\gamma_i)}\lam_{\odd}^{-b_{\odd}(\gamma_i)}.$
\end{proof}

\begin{prop}
  \label{prop:contour_partfn}
  Let $\Lambda$ be a finite induced subgraph of $G$. Then
  \begin{equation*}
    Z^{\even}_{\Lambda} = \lam_{\even}^{|V_\even \cap \Lambda|}\sum_{\substack{\Gamma \subseteq \cCb(\Lambda) \\ \text{compatible} \\ \text{matched} \\\text{external}\ \even}} \prod_{\gamma \in \Gamma} w_\gamma.
\end{equation*}
The same holds with odd replacing even. 
\end{prop}
\begin{proof}
The result follows from Propositions~\ref{lem:bijection} and~\ref{lem:contour_weights}.
\end{proof}

In Section~\ref{sec:convergence} we will need to control the weights
of contours based on their \emph{size} $|\gamma|$, where $|\gamma|$ is
the number of edges in $\gamma$. This requires comparing
$b_{\even}(\gamma)$ and $b_{\odd}(\gamma)$. We begin with a lemma that
will suffice for graphs that are vertex transitive within a parity
class.
\begin{lemma}
  \label{lem:VTPC}
  Suppose $G$ is a bipartite graph with all even vertices having
  degree $\Delta_{\even}$ and odd vertices having degree
  $\Delta_{\odd}$. Then
  $b_{\even}(\gamma) = \frac{|\gamma|}{\Delta_{\even}} +
  \frac{\Delta_{\odd}}{\Delta_{\even}} |\inte_{\odd}\gamma \cap
  V_{\odd}|$ for $\gamma$ even, and
  $b_{\odd}(\gamma) = \frac{|\gamma|}{\Delta_{\odd}} +
  \frac{\Delta_{\even}}{\Delta_{\odd}} |\inte_{\even}\gamma \cap
  V_{\even}|$ for $\gamma$ odd.
\end{lemma}
\begin{proof}
  We consider the case of $\gamma$ even; the odd case is
  analogous. Observe that
  \begin{equation*}
    b_{\even}(\gamma) =
    \frac{1}{\Delta_{\even}}\sum_{v\in\inte_{\odd}\gamma\cap
      V_{\even}} \Delta_{\even}
    =
    \frac{|\gamma|}{\Delta_{\even}} + \frac{1}{\Delta_{\even}}
    \sum_{w\in\inte_{\odd}\gamma \cap V_{\odd}}\Delta_{\odd},
  \end{equation*}
  where the second equality has used that if $w$ is adjacent to $v$,
  then either $\{v,w\}$ is an edge in $\gamma$ or else $w\in
  \inte_{\odd}\gamma\cap V_{\odd}$. The factor $\Delta_{\odd}$
  accounts for the number of $v$ adjacent to a given $w$.
\end{proof}

The next lemma will handle the vertex transitive setting. Recall that
our results in this setting concern $\lam_{\even}=\lam_{\odd}=\lam$,
so we set $b(\gamma) = b_{\even}(\gamma)+b_{\odd}(\gamma)$.
\begin{lemma}
  \label{lem:b_transitive}
  Suppose $G$ is a $\Delta$-regular graph, and let $\gamma$ be a
  contour of $G$. Then $b(\gamma) = |\gamma|/\Delta$.
\end{lemma}
\begin{proof}
  Apply Lemma~\ref{lem:VTPC}, using that $\Delta_{\even}=\Delta_{\odd}=\Delta$.
\end{proof}

Lastly we consider matched automorphic graphs.
\begin{lemma}
  \label{lem:b_matched}
  If $G$ is matched automorphic and has maximum degree $\Delta$, then
  $b(\gamma) \geq |\gamma| / \Delta$.
\end{lemma}
\begin{proof}
  Let $\pi$ be a matched automorphism of $G$. Without loss of
  generality, we assume $\gamma$ is an even contour.  We will consider
  both $\inte_\odd\gamma\cap V_\odd$ and
  $\inte_\odd\gamma \cap V_\even$, but will divide the latter into two
  sets: The set $S_1$ of vertices $v\in \inte_\odd\gamma \cap V_\even$
  such that $\pi(v)$ is not in $\inte_\odd\gamma$, and the set $S_2$
  of vertices $v \in \inte_\odd\gamma \cap V_\even$ such that $\pi(v)$
  is in $\inte_\odd\gamma$. Because all vertices adjacent to
  $\partial \inte_\odd\gamma = \gamma$ are even, by counting all edges
  within $\inte_\odd\gamma$ in two different ways we see that
  \begin{equation*}
    \sum_{v \in \inte_\odd\gamma \cap V_\odd} \mathrm{deg}(v) = \sum_{v \in \inte_\odd\gamma \cap V_\even} \mathrm{deg}(v) - |\partial \inte_\odd\gamma|
  \end{equation*}
  Rearranging terms and splitting a sum into terms for $S_1$ and
  $S_2$,
  \begin{align*}
    |\gamma| = |\partial \inte_\odd\gamma|
    &= \sum_{v \in \inte_\odd\gamma \cap V_\even} \mathrm{deg}(v) - \sum_{v \in \inte_\odd\gamma \cap V_\odd} \mathrm{deg}(v) \\
    & = \sum_{v \in S_1} \mathrm{deg}(v) + \sum_{v \in S_2}
      \mathrm{deg}(v) - \sum_{v \in \inte_\odd\gamma \cap V_\odd} \mathrm{deg}(v) \, .
  \end{align*}
  Because all vertices $v \in \inte_\odd\gamma \cap V_\odd$ must have
  $\pi^{-1}(v) \in \inte_\odd\gamma$, and $v$ and $\pi(v)$ must have
  the same degree, this becomes
  \begin{equation*}
    |\gamma|  
    = \sum_{v \in S_1} \mathrm{deg}(v) \leq \Delta |S_1|
  \end{equation*}
  This implies
  $b(\gamma) = |\inte_\odd\gamma \cap V_\even| - |\inte_\odd\gamma
  \cap V_\odd| = |S_1| \geq |\gamma|/\Delta$, as desired.
\end{proof}

\subsection{External Contour Representation and Polymer Representation}
\label{sec:contour-polym-repr}

This section transforms the contour representation of an independent
set into a polymer model representation. The basic definitions of the
latter are given in Section~\ref{sec:polymer}, and the transformation
is carried out in Section~\ref{sec:PS}.

\subsubsection{Polymer Models}
\label{sec:polymer}

We briefly recall the setup of abstract polymer
systems~\cite{kotecky1986cluster}.  A \emph{polymer model} consists of
three things. First, a set $\cP$ of \emph{polymers}. Second, a
pairwise symmetric compatibility relation on $\cP$, denoted by
$\gamma \sim \gamma'$. Incompatibility is denoted by
$\gamma \not\sim \gamma'$, and self-incompatibility is required, i.e.,
$\gamma \nsim \gamma$ for all $\gamma \in \cP$. Lastly, there is a
weight $w\colon \cP\to \mathbb{C}$; we denote the weight of
$\gamma\in\cP$ by $w_{\gamma}$.  Suppose $\cP$ is finite. The
\emph{polymer partition function} for $\cP$ is then
\begin{equation*}
\Xi(\cP)= \sum_{\substack{\Gamma \subseteq \cP \\ \text{compatible}}}
\prod_{\gamma \in \Gamma} w_{\gamma}  
\end{equation*}
where the sum is over all pairwise compatible collections of polymers.
The empty collection of polymers contributes $1$ to the sum.  

The next section will show how $Z^{\even}_{\Lambda}$ and
$Z^{\odd}_{\Lambda}$ can be written as polymer model partition
functions, with polymers being contours. This is useful as
  there are well-known criteria for establishing convergent expansions
  for $\log \Xi$, as will be recalled in Section~\ref{sec:CEPoly}.

\subsubsection{External Contour Representation and Polymer Representation}
\label{sec:PS}

For $\Lambda$ a finite induced subgraph of
$G$ with $\partial\Lam=\binf\Lam$, recall the representation of the partition function
$Z^{\even}_{\Lambda}$ given by Proposition~\ref{prop:contour_partfn}:
  \begin{equation*}
    Z^{\even}_{\Lambda} 
    =
    \lambda_{\even}^{|V_\even \cap \Lambda|} {\sum_{\Gamma\subset
        \cCb(\Lambda)}}^{\!\!\!\!\!\even}
  \   \prod_{\gamma \in
      \Gamma} w_{\gamma},
\end{equation*}
where the notation $\sum^{\even}$ denotes a sum over compatible,
matched, external even contours. We use $\sum^{\odd}$ analogously.

It will be convenient to work with a normalized version of
$Z^{\even}_{\Lambda}$:
\begin{equation}
  \label{eq:Xie}
    \Xi^{\even}_{\Lambda} = \frac{Z^{\even}_{\Lambda} }{\lambda_{\even}^{|V_\even \cap \Lambda|}} = {\sum_{\Gamma\subset
        \cCb(\Lambda)}}^{\!\!\!\!\!\even}
  \  \prod_{\gamma \in
      \Gamma} w_{\gamma}.
\end{equation}
Note that $\Xi^{\even}_\Lambda$ has leading term $1$ when viewed as a
polynomial in variables $w_{\gamma}$ by
Proposition~\ref{lem:bijection}.  This expression for
$\Xi^{\even}_\Lambda$ is not yet a polymer partition function.  This
is because (recall Section~\ref{sec:bij}) the matching condition
placed on collections of contours is not a pairwise condition.  In
this section we adapt the standard approach to circumventing this
issue: we rewrite $\Xi^{\even}_\Lambda$ as a sum over compatible even
contours, with no further constraints.  This paves the way for proving
that the resulting polymer models have convergent expansions in
Section~\ref{sec:convergence}.

We begin with two preparatory lemmas that show that resumming contours
contained inside a contour $\gamma$ yields partition functions with
appropriate boundary conditions.

\begin{lemma}
  \label{lem:compat-internal-odd}
  Let $\gamma$ and $\gamma'$ be two contours with
  $\gamma'\prec\gamma$.  Then $\gamma'$ is compatible with $\gamma$ if
  and only if $\gamma'$ is in
  $\cCb(\inte \gamma)=\cCb(\inte_{\odd} \gamma)\cup \cCb(\inte_{\even}
  \gamma)$.
\end{lemma}
\begin{proof}
  Suppose $\gamma'$ is compatible with $\gamma$.  By
  Lemma~\ref{lem:order}, all of the endpoints of edges of $\gamma'$
  are contained in $\inte_{\even}(\gamma)$ or $\inte_{\odd}(\gamma)$,
  i.e., $\gamma'\in \cC(\inte_{\even} \gamma)$ or
  $\gamma'\in\cC(\inte_{\odd}\gamma)$. We consider the case
  $\gamma'\in \cC(\inte_{\odd}\gamma)$; the other case is exactly
  analogous.  Compatibility means there is no basis cycle containing
  both an edge of $\gamma$ and an edge of $\gamma'$. Because
  $\partial \inte_{\odd} \gamma \subseteq \gamma$, this means that
  $\gamma'$ cannot contain any edge of a basis cycle that leaves
  $\inte_{\odd} \gamma$, and hence
  $\gamma' \in \cCb(\inte_{\odd} \gamma)$.

  If $\gamma' \in \cCb(\inte_{\odd} \gamma)$, then $\gamma'$ cannot
  contain any edges that are part of a basis cycle that exits
  $\inte_{\odd} \gamma$. On the other hand, no edge of $\gamma$ can
  have both endpoints in $\inte_{\odd}\gamma$, as these endpoints
  would have the same parity by the definition of an odd component. So
  every basis cycle containing an edge of $\gamma$ exits
  $\inte_{\odd}(\gamma)$.  This shows $\gamma$ and $\gamma'$ are
  compatible. The case $\gamma' \in \cCb(\inte_{\even} \gamma)$ is
  identical.
\end{proof}

Given a contour $\gamma$, let $\cM_{\gamma}$ denote the set of
collections of contours $\Gamma$ such that (i) for all
$\gamma'\in\Gamma$, $\gamma'\prec\gamma$ and (ii) $\Gamma\cup\gamma$
is a matched and compatible collection of contours.

\begin{lemma}
  \label{lem:internal_xi}
  Let $\Lambda$ be a finite induced subgraph of $G$ with
  $\partial\Lam=\binf\Lam$. Then for any contour
  $\gamma \in \overline{\cC}( \Lambda)$,
  \begin{equation}
   \mathop{\sum_{\Gamma\subseteq \cCb(\Lambda)}}_{\Gamma\in
     \cM_{\gamma}} \prod_{\gamma' \in \Gamma}  w_{\gamma'} =
   \Xi^{\even}_{\inte_{\even} \gamma} \Xi^{\odd}_{\inte_{\odd}
     \gamma}. 
\end{equation}
\end{lemma}
\begin{proof}
  By Lemma~\ref{lem:compat-internal-odd}, 
    \begin{equation*}
     \mathop{\sum_{\substack{\Gamma\subseteq \cCb(\Lambda)\\ \Gamma\in \cM_{\gamma}}}}
    \prod_{\gamma' \in \Gamma} w_{\gamma'} 
    = \left( \sum_{\substack{\Gamma\subseteq \cCb(\inte_{\even}(\gamma))\\ \Gamma\in \cM_{\gamma}}}\prod_{\gamma' \in
        \Gamma} w_{\gamma'} \right) \left( \sum_{\substack{\Gamma\subseteq \cCb(\inte_{\odd}(\gamma))\\ \Gamma\in \cM_{\gamma}}}\prod_{\gamma' \in \Gamma} w_{\gamma'}\right)  
\end{equation*}
For a collection of compatible contours $\Gamma$ in
$\cCb(\inte_{\even}\gamma)$, $\Gamma\cup\gamma$ matching simply means
that all external contours of $\Gamma$ are even.  By
Proposition~\ref{prop:contour_partfn}, this is the same as enforcing
an even boundary condition on $\inte_{\even} \gamma$. Thus
\begin{equation*}
   \sum_{\substack{\Gamma\subseteq \cCb(\inte_{\even}(\gamma))\\ \Gamma\in
     \cM_{\gamma}}}\prod_{\gamma' \in \Gamma} w_{\gamma'}  =
   \Xi^{\even}_{\inte_{\even} \gamma} 
\end{equation*}
The same holds for $\inte_{\odd} \gamma$ and $\Xi^{\odd}_{\inte_{\odd}
  \gamma}$, proving the lemma.
\end{proof}

Next we reformulate $\Xi^{\even}_\Lambda$ in a manner more convenient
for analysis. Let $\cCb^{\even}(\Lambda)$ denote the subset of even
contours in $\cCb(\Lambda)$. We say a compatible collection
$\Gamma\subset\cCb(\Lambda)$ is \emph{mutually external} if each
$\gamma\in\Gamma$ is external for $\Gamma$.  Write
$\sum^{\even,\ext}_{\cCb(\Lambda)}$ to denote a sum over collections
of compatible and mutually external contours in
$\cCb^{\even}(\Lambda)$. Lastly, define
\begin{equation}
  \label{eq:weo}
  \widetilde{w}_{\gamma} =
  w_\gamma\dfrac{\Xi^{\odd}_{\inte_{\odd} \gamma}}{\Xi^{\even}_{\inte_{\odd}
        \gamma}} \quad
    \text{if $\gamma$ is even,}
    \qquad
    \widetilde{w}_{\gamma} = 
    w_\gamma\dfrac{\Xi^{\even}_{\inte_{\even} \gamma}}{\Xi^{\odd}_{\inte_{\even}
        \gamma}} \quad
 \text{if $\gamma$ is odd}.
\end{equation}

\begin{lemma}
  \label{lem:excon}
  Let $\Lambda$ be a finite induced subgraph of $G$ with $\partial
  \Lambda = \binf \Lambda$. Then 
    \begin{equation*}
  \Xi_{\Lambda}^{\even}  =  {\sum_{\Gamma\subset
        \cCb(\Lambda)}}^{\!\!\!\!\!\even,\ext} \prod_{\gamma \in
      \Gamma} \widetilde{w}_\gamma \Xi^{\even}_{\inte_{\even} \gamma}
    \Xi^{\even}_{\inte_{\odd} \gamma}, \qquad \text{and} \qquad
    \Xi_{\Lambda}^{\odd}  =   {\sum_{\Gamma\subset
        \cCb(\Lambda)}}^{\!\!\!\!\!\odd,\ext} \prod_{\gamma \in \Gamma} \widetilde{w}_\gamma \Xi^{\odd}_{\inte_{\even} \gamma} \Xi^{\odd}_{\inte_{\odd} \gamma}. 
  \end{equation*}
\end{lemma}
\begin{proof}
  We consider the even case; the odd case is analogous. Note that by
  Lemma~\ref{lem:spo_order} (see \eqref{eq:product}) collections of
  compatible contours correspond to collections of mutually external
  contours, together with contours interior to these external
  contours. Thus by (i) grouping terms in the
  expression~\eqref{eq:Xie} of $\Xi_{\Lambda}^{\even}$ according to
  their external contours, which must be even because of the boundary
  conditions, and (ii) applying Lemma~\ref{lem:internal_xi}, we obtain
\begin{align*}
  \Xi_{\Lambda}^{\even}
  &=  {\sum_{\Gamma\subset
        \cCb(\Lambda)}}^{\!\!\!\!\!\even} w_\gamma
  \\  & =   {\sum_{\Gamma\subset
        \cCb(\Lambda)}}^{\!\!\!\!\!\even,\ext}\left( w_\gamma
  \sum_{\Gamma' \subseteq \cCb(\Lambda),\Gamma'\in\cM_{\gamma}} \prod_{\gamma'\in\Gamma'}w_{\gamma'}\right)   
    \\ &=    {\sum_{\Gamma\subset
        \cCb(\Lambda)}}^{\!\!\!\!\!\even,\ext}    \prod_{\gamma \in \Gamma} w_\gamma \Xi^{\even}_{\inte_{\even} \gamma} \Xi^{\odd}_{\inte_{\odd} \gamma} 
    \\ &=  {\sum_{\Gamma\subset
        \cCb(\Lambda)}}^{\!\!\!\!\!\even,\ext} \prod_{\gamma \in \Gamma} \left( w_\gamma\frac{\Xi^{\odd}_{\inte_{\odd} \gamma}}{\Xi^{\even}_{\inte_{\odd} \gamma}}\right) \Xi^{\even}_{\inte_{\even} \gamma} \Xi^{\even}_{\inte_{\odd} \gamma} .\qedhere
\end{align*}
\end{proof}

Lemma~\ref{lem:excon} is called the \emph{external contour
  representation}, and we can use it to derive a polymer model
formulation for $\Xi_{\Lambda}^{\even}$.
\begin{lemma}
  \label{lem:polyrep}
  Let $\Lambda$ be a finite induced subgraph of $G$ with
  $\partial\Lambda=\binf\Lam$. Then
  \begin{equation}
    \label{eq:polyrep}
    \Xi_{\Lambda}^{\even}   =  \sum_{\substack{\Gamma \subseteq
        \cCb^{\even}(\Lambda) \\compatible}} \prod_{\gamma \in \Gamma}
    \widetilde{w}_\gamma, \qquad 
        \Xi_{\Lambda}^{\odd}   =  \sum_{\substack{\Gamma \subseteq
        \cCb^{\odd}(\Lambda) \\compatible}} \prod_{\gamma \in \Gamma}
    \widetilde{w}_\gamma.
  \end{equation}
\end{lemma}
Note that Lemma~\ref{lem:polyrep} achieves what has been
promised. Compatibility is a pairwise condition on contours, so the
right-hand side of~\eqref{eq:polyrep} is a polymer model partition
function. Note, however, that there is no longer a bijection between
the contour configurations contributing to $\Xi_{\Lambda}^{\even}$ and
the independent sets contributing to $Z^{\even}_{\Lambda}$.

\begin{proof}[Proof of Lemma~\ref{lem:polyrep}]
  We prove this by induction on $|\cCb(\Lambda)|$ in the even
  case. The odd case is analogous. First, suppose that
  $|\cCb(\Lambda)| = 0$.  The only possible
  $\Gamma \subseteq \cCb^{\even}(\Lambda)$ is the empty set, and the
  empty product $\prod_{\gamma \in \Gamma}$ is 1. It follows that both
  sides of the equations in~\eqref{eq:polyrep} are 1.

  For the inductive step, suppose $\Lambda$ is such that
  $|\cCb(\Lambda)|=K\geq 1$, and that~\eqref{eq:polyrep} holds for all
  $\Lambda'$ with $|\cCb(\Lambda')|<K$.  Note that for any
  $\gamma \in \cCb(\Lambda)$, this inductive hypothesis applies to
  $\Lambda' = \inte_{\even} \gamma$ and
  $\Lambda' = \inte_{\odd} \gamma$ by
  Lemma~\ref{lem:int_ind}. Combining this with Lemma~\ref{lem:excon},
  we obtain
  \begin{align*}
    \Xi_{\Lambda}^{\even}
    &=
      {\sum_{\Gamma\subset
      \cCb(\Lambda)}}^{\!\!\!\!\!\even,\ext}
      \prod_{\gamma \in \Gamma} \widetilde{w}_\gamma
      \Xi^{\even}_{\inte_{\even} \gamma} \Xi^{\even}_{\inte_{\odd} \gamma}
    \\
    &=
      {\sum_{\Gamma\subset
      \cCb(\Lambda)}}^{\!\!\!\!\!\even,\ext}
      \prod_{\gamma \in \Gamma} \widetilde{w}_\gamma
      \left(\sum_{\substack{\Gamma' \subseteq \cCb^{\even}(\inte_{\even} \gamma)
    \\ \text{compatible}}} \prod_{\gamma' \in \Gamma'} \widetilde{w}_{\gamma'}
    \right)
    \left( \sum_{\substack{\Gamma' \subseteq \cCb^{\even}(\inte_{\odd} \gamma)
    \\\text{compatible}}} \prod_{\gamma' \in \Gamma'}  \widetilde{w}_{\gamma'} \right)  .
  \end{align*}

  By Lemma~\ref{lem:compat-internal-odd}, an even contour $\gamma'$ in
  $\inte_{\even} 
  \gamma$ is compatible with $\gamma$ if and only if it is in
  $\cCb^{\even}(\inte_{\even} \gamma)$. The same statement is true
  when replacing $\inte_{\even}\gamma$ with $\inte_{\odd}\gamma$.  We
  conclude
  \begin{equation}
    \Xi_{\Lambda}^{\even} = \sum_{\substack{\Gamma \subseteq \cCb^{\even}(\Lambda) \\ \text{compatible} }} \prod_{\gamma \in \Gamma} \widetilde{w}_\gamma
  \end{equation}
  which advances the induction.
\end{proof}

\section{Pirogov--Sinai Theory: Convergence}
\label{sec:convergence}

This section carries out the main analytic step of Pirogov--Sinai
theory, an inductive argument that controls the weights of the polymer
models from Section~\ref{sec:contour-polym-repr}.  This requires the following
strengthening of Assumption~\ref{as:1}. 

\begin{assumption}
  \label{as:2}
  $G$ is an infinite connected one-ended bipartite graph with maximum
  degree $\Delta$, and $\cB$ is a $D$-bounded cycle basis of $G$.
\end{assumption}

A crucial consequence of Assumption~\ref{as:2} is the following bound
on the number of contours of a given size. This is a key place where
we make use of having a bounded cycle basis.
\begin{lemma}
  \label{lem:enum_contours}
  Let $G$ be a graph with a $D$-bounded cycle basis $\cB$.  The number
  of contours of size $k$ containing a given edge $e$ is at most
  $(eD)^{k-1}$.
\end{lemma}
\begin{proof}
  Let $G_\cB$ be the graph whose vertices are the edges of $G$ and
  whose edge set connects edges of $G$ that are in a common basis
  cycle in $\cB$. This graph has maximum degree $D$, and a contour can
  be identified with a connected induced subgraph of $G_\cB$. The
  result now follows from well known enumerations of such subgraphs,
  see, e.g.,~\cite{BorgsChayesKahnLovasz}.
\end{proof}

\subsection{Cluster Expansion Preliminaries}
\label{sec:CEPoly}

Our main tool to gain analytic control of the polymer models will be
the cluster expansion.  The next subsection recalls the cluster expansion
convergence criteria from~\cite{kotecky1986cluster}. The subsequent
subsection then applies this criteria to derive the estimates that we
will need in the sequel.

\subsubsection{Cluster Expansion for Polymer Models}
\label{sec:cepm}
Consider the setting of polymer models as described in
Section~\ref{sec:contour-polym-repr}.  The \emph{cluster expansion} is an
infinite series representation of $\log \Xi (\cP)$. Given an ordered
multiset $X$ of polymers, the \emph{incompatibility graph} $H(X)$ of
$X$ has a vertex for every polymer and an edge between each pair of
incompatible polymers. A \emph{cluster} is an ordered multiset $X$ of
polymers from $\cP$ such that $H(X)$ is connected. If $\cX$ is the set
of all clusters from $\cP$, then as a formal power series in the
variables $w_{\gamma}$, the cluster expansion states
\begin{equation}
  \label{eq:CEformal}
  \log \Xi (\cP) = \sum_{X \in \cX} w(X) \, ,
\end{equation}
where
\begin{equation}
  \label{eq:CEterms}
  w(X) = \phi(H(X)) \prod_{\gamma \in X} w_{\gamma}, \qquad  \phi(H) = \frac{1}{|V(H)|!} \sum_{\substack{ A \subseteq E(H) \\ \text{spanning, connected} }} (-1)^{|A|} \,.
\end{equation}
The function $\phi(H)$ is the \emph{Ursell
  function}. Equation~\eqref{eq:CEformal} is only an equality of
formal power series. A sufficient condition for this equality to hold
analytically, i.e., with the right-hand side being an absolutely
convergent power series, is the following \emph{Koteck\'{y}--Preiss
  condition}.
\begin{theorem}[\cite{kotecky1986cluster}]
  \label{thmKP}
  Let $\alpha_1 \colon \cP \to [0,\infty)$
  and $\alpha_2\colon \cP \to [0,\infty)$ be two given functions and
  suppose that for all $\gamma \in \cP$,
  \begin{align}
    \label{eqKPcond1}
    \sum_{\gamma' \nsim \gamma}  |w_{\gamma'} |e^{\alpha_1(\gamma') + a_2(\gamma')}  &\le \alpha_1 (\gamma) \,.
  \end{align}
  Then the cluster expansion for the polymer model defined by
    any finite subset of polymers converges absolutely. Moreover, for
  all $\gamma \in \cP$ we have
  \begin{equation}
    \label{eqKPtail1}
    \sum_{\substack{X \in \cX :\\ \exists \gamma' \in X, \gamma' \nsim
        \gamma }}   |w(X) | \prod_{\gamma' \in X} e^{\alpha_2(\gamma')} \le
    \alpha_1(\gamma) \,. 
  \end{equation}
\end{theorem}

\subsubsection{Cluster Expansion Convergence for Contour Models}
This section shows that Assumption~\ref{as:2} suffices to conclude the
cluster expansion converges for the specific polymer models defined in
Section~\ref{sec:contour-polym-repr}. Given this,
Proposition~\ref{prop:FV729} then summarizes estimates that follow
from this convergence.

Proposition~\ref{prop:FV729} is broadly similar
to~\cite[Theorem~7.29]{friedli2017statistical}, but as our setting
does not have a notion of translation (or translation invariance), the
division of the logarithms of partition functions into bulk and
surface contributions is formulated differently. Thus we begin by
introducing some notation.  Recall from Section~\ref{sec:cepm} that we
write $\cX=\cX(\cP)$ to denote the set of clusters associated to a
polymer model $\cP$. For $X\in\cX$ write $\overline{X}$ to denote the
support of $X$, i.e., the union of the edge sets of the polymers in
$X$. Set
\begin{equation}
  \label{eq:FV729bulk}
  Q(v) = \sum_{u\in N(v)}\mathop{\sum_{X\in \cX}}_{\{u,v\}\in
    \overline{X}} \frac{1}{|\overline{X}|} w(X). 
\end{equation}
For $\cP\subset \cC(G)$ let
$\cP_{\Lambda} = \cP\cap \overline{\cC}(\Lambda)$, and define
\begin{equation}
  \label{eq:FV729SurfaceS}
  S_{\odd}(\partial\Lambda)= \sum_{v\in \Lambda\cap V_{\odd}}\sum_{u\in N(v)}
  \sum_{X\in\cX(\cP)\setminus \cX(\cP_{\Lambda})}
  \frac{1_{\{u,v\}\in\bar X}}{|\overline{X}|} w(X).   
\end{equation} 
Then, as formal power series, we claim that 
\begin{equation}
  \label{eq:FV729Surface}
  \log \Xi (\cP_{\Lambda}) = \sum_{v\in
    \Lambda \cap V_{\odd}}Q(v) - S_{\odd}(\partial\Lambda).
\end{equation}
Verifying~\eqref{eq:FV729Surface} is a matter of applying
Theorem~\ref{thmKP} and re-arranging: since every contour in
$\cP_{\Lambda}$ contains at least one edge in~$\Lambda$,
\begin{equation}
  \label{eq:FV729Surface-odd}
      \log \Xi (\cP_{\Lambda})
  = \sum_{X\in\cX(\cP_{\Lambda})} w(X)
  = \sum_{\{u,v\}\in E(\Lambda)}
    \mathop{\sum_{X\in\cX(\cP_{\Lambda})}}_{\{u,v\}\in \overline
    X}\frac{w(X)}{|\overline X|} 
  = \sum_{v\in \Lambda\cap V_{\odd}} \sum_{u\in N(v)}
    \mathop{\sum_{X\in\cX(\cP_{\Lambda})}}_{\{u,v\}\in \overline X}\frac{w(X)}{|\overline X|},
\end{equation}
where in the last equality we have used that if
  $\{u,v\}\notin E(\Lambda)$ then $\{u,v\}\notin \overline{X}$ for any
  cluster $X\in\cX(\cP_{\Lambda})$.  The
claim~\eqref{eq:FV729Surface} then follows by rewriting the
sum over $\cX(\cP_{\Lambda})$ as the difference of the sums over
$\cX(\cP)$ and $\cX(\cP)\setminus \cX(\cP_{\Lambda})$. The same
argument shows that we can also write, with $S_{\even}$ defined by
replacing $V_{\odd}$ by $V_{\even}$ in the
formula~\eqref{eq:FV729SurfaceS},
\begin{equation}
  \label{eq:FV729Surface-even}
   \log \Xi (\cP_{\Lambda}) = \sum_{v\in
    \Lambda \cap V_{\even}}Q(v) - S_{\even}(\partial\Lambda).
\end{equation}

\begin{prop}
  \label{prop:FV729}
  Suppose $G$ satisfies Assumption~\ref{as:2}, and consider a polymer
  model with polymer set $\mathcal{P}\subset \mathcal{C}(G)$ and
  weights $\bar w\colon \mathcal{P}\to \mathbb{C}$ with $\bar w$
  depending on a real parameter $s\in \ob{a,b}$. Consider the following
  two hypotheses. First, 
  \begin{equation}
    \label{eq:FV729Hw}
    |\bar w(\gamma)| \leq e^{-\tau |\gamma|}, \qquad (s\in \ob{a,b}).
  \end{equation}
  Second, $\bar w(\gamma)$ is continuously differentiable in
  $s\in \ob{a,b}$, and there are $\tau,R>0$ such that 
  \begin{equation}
    \label{eq:FV729Hd}
    \left|\frac{\textrm{d}\bar w(\gamma)}{\textrm{d}s}\right| \leq R
    |\gamma|^{2} e^{-\tau |\gamma|}, \qquad (s\in \ob{a,b}).
  \end{equation}

  \begin{enumerate}[noitemsep]
  \item Under the hypothesis~\eqref{eq:FV729Hw}, there is a
  $\tau_{1}(D)$ such that if $\tau>\tau_{1}$, then for
  any vertex~$v$, $|Q(v)|\leq \eta(\tau)=e^{-\tau/3}$. Moreover,
  $|S_{\odd}(\partial\Lambda)|, |S_{\even}(\partial\Lambda)| \leq
  \eta |\partial\Lambda|$.
  \item If in addition~\eqref{eq:FV729Hd} holds, then there is a (possibly
  larger) $\tau_{1}(D,R)$ such that if $\tau>\tau_{1}$
  then for each vertex $v$, $Q(v)$ is continuously differentiable in
  $s\in (a,b)$, it's derivative is given by the sum of the derivatives
  of the summands of~\eqref{eq:FV729bulk}, and is at most
  $Re^{-\tau/3}$ in magnitude. Moreover, for any edge
    $\{u,v\}$ and $L\geq \min_{\gamma\in\cP}|\gamma|$
    \begin{equation}
      \label{eq:FV731}
      \left| \sum_{\substack{X \in \cX \\ \{u,v\} \in \overline{X}}}
          \frac{1}{|\ov{X}|}   w(X) 1_{|\ov{X}| \geq L} \right| \leq e^{-\tau L/2}.
    \end{equation}
  \end{enumerate}
\end{prop}

\begin{proof}
  We start with 1. Towards applying Theorem~\ref{thmKP}, set
  $\alpha_{1}(\gamma)=|\gamma|$ and
  $\alpha_{2}(\gamma)=\frac{2\tau}{3}|\gamma|$.  For any edge $e$ of
  $G$, write $e\nsim\gamma$ if there is a basis cycle that contains
  both $e$ and an edge of $\gamma$. Note that for any contour $\gamma$
  the number of edges $e\nsim\gamma$ is at most $D|\gamma|$ since
  Assumption~\ref{as:2} gives a $D$-bounded cycle basis of
  $G$. Lemma~\ref{lem:enum_contours} implies the number of $\ell$-edge
  contours that contain a fixed edge $e$ is at most
  $(eD)^{\ell-1}$. Hence for any fixed contour $\gamma$,
  \begin{align*}
    \label{eq:KPverify}
    \mathop{\sum_{\gamma'\in\cP}}_{\gamma'\nsim \gamma}
    \overline{w}(\gamma') e^{\alpha_{1}(\gamma')+\alpha_{2}(\gamma')}
    &\leq
    \mathop{\sum_{e\in E(G)}}_{e\nsim\gamma}
      \mathop{\sum_{\gamma'\in\cP}}_{e\in \gamma'}
      e^{-\tau |\gamma'| + \alpha_{1}(\gamma') + \alpha_{2}(\gamma')}
    \\
    &\leq
      D|\gamma| \sum_{\ell=1}^{\infty} (eD)^{\ell-1} e^{-\tau
      \ell + \ell+ \frac{2\tau}{3}\ell} \\ 
    &\leq |\gamma| \sum_{\ell=1}^{\infty} (e^{-\frac{\tau}{3}}e^{2}D)^{\ell},
  \end{align*}
  which is at most $|\gamma| = \alpha_1(\gamma)$ if
  $\tau$ is large enough. This verifies the hypothesis of
  Theorem~\ref{thmKP}. The estimates on
  $Q(v)$, $S_\odd(\partial\Lambda)$ and
  $S_{\even}(\partial\Lambda)$ follow by using~\eqref{eqKPtail1}, as
  the factors of $\alpha_2(\gamma')$ give the desired
  decay. 

  For 2.\ the continuous differentiability of $Q(v)$ under~\eqref{eq:FV729Hd},
  as well as the subsequent estimates on $Q(v)$, are obtained as in the proof
  of~\cite[Lemma~7.29]{friedli2017statistical} -- these proofs rely
  only on the uniformity of the estimates on polymer weights (and
  their derivatives). Lastly, the estimate~\eqref{eq:FV731} can be
  obtained as in the proof of~\cite[Lemma~7.31]{friedli2017statistical}.
\end{proof}

\subsection{Transitive and Matched Automorphic Graphs}
\label{sec:conv_trans_autom}

In this section we prove the convergence of the cluster expansion for
the polymer model representations of $\Xi^{\even}_{\Lambda}$ and
$\Xi^{\odd}_{\Lambda}$ given by Lemma~\ref{lem:polyrep} when
  $G$ is transitive or matched automorphic. The proof is
inductive. Towards this, let $\cC^{\even}_{k}$ denote the subset of
even contours with $|\inte\gamma|\leq k$, and define
\begin{equation}
  \label{eq:Qek}
  Q^{\even}_{k}(v) = \sum_{u \in N(v)} \sum_{\substack{X \in \cX(\cC^{\even}_{k}) \\ \{u,v\} \in \overline{X}}}  \frac{1}{|\overline{X}|} \widetilde{w}(X);
\end{equation}
i.e., $Q^{\even}_{k}$ is as in~\eqref{eq:FV729bulk} with the choice $\cP =
\cC^{\even}_{k}$. Define $Q^{\odd}_{k}$ analogously.

The key step of Pirogov-Sinai theory involves controlling the ratio
$\Xi^{\odd}_{\inte_{\odd} \gamma}/\Xi^{\even}_{\inte_{\odd}
  \gamma}$. The symmetries present in the transitive and matched
automorphic setting imply that this ratio exhibits a great deal of
cancellation: volume factors depending exponentially on
$|\inte \gamma|$ cancel, leaving only surface factors.  We first
exhibit the cancellations in both cases. Note that the
  additional assumption on the cycle basis in the next two lemmas does
  not restrict the class of graphs being considered, recall
  Lemma~\ref{lem:BAI}.

\begin{lemma}
  \label{lem:volume-cancel-transitive}
  Suppose $G$ is a transitive graph with a transitive cycle basis that
  satisfies Assumption~\ref{as:2}. Then for all $v \in V_\even$ and
  $v' \in V_\odd$, $Q^\even_k(v) = Q^\odd_k(v')$. Call this common
  value $Q_{k}$. If $Q_k \leq Q$, then for any even contour~$\gamma$,
  \begin{equation}
    \left|\sum_{v \in \inte_\odd \gamma \cap V_\even} Q^\even_k(v) - \sum_{v' \in \inte_\odd \gamma \cap V_\odd} Q^\odd_k(v') \right| \leq Q |\gamma|.
  \end{equation}
  An equivalent statement holds for odd contours $\gamma$ and their
  even interiors $\inte_\even(\gamma)$.
\end{lemma}
\begin{proof}
  That $Q^\even_k(v)=Q^\odd_k(v')$ for all $v \in V_\even$,
  $v' \in V_\odd$ follows from the defining formulas since $G$ is
  transitive with a transitive cycle basis. Call this common value
  $Q_k$.  By Lemma~\ref{lem:b_transitive},
  \begin{equation}
    \left|\sum_{v \in \inte_\odd \gamma \cap V_\even} Q^\even_k(v) -
      \sum_{v' \in \inte_\odd \gamma \cap V_\odd} Q^\odd_k(v') \right|
    =b(\gamma)\left| \ Q_k \right|
    \leq \frac{Q |\gamma|}{\Delta} \leq Q|\gamma|\, . \qedhere 
  \end{equation}
\end{proof}

\begin{lemma}
  \label{lem:volume-cancel-matched}
  Suppose a graph $G$ satisfying Assumption~\ref{as:2} has a matched
  automorphism $\pi$ such that $\cB$ is invariant under $\pi$.
  Suppose $Q$ is such that for all $v \in V_\even$,
  \begin{equation}
    |Q^\even_k(v)| \leq Q. 
  \end{equation}
  Then for any contour $\gamma$,
  \begin{equation}
    \left| \sum_{v \in \inte_\odd \gamma \cap V_\even} Q^\even_k(v) -
      \sum_{v' \in \inte_\odd \gamma \cap V_\odd} Q^\odd_k(v') \right|
    \leq Q |\gamma|. 
  \end{equation}
  The same result holds when replacing $\inte_{\odd}\gamma$ with
  $\inte_{\even}\gamma$.
\end{lemma}
\begin{proof}
  For any $v \in V_\even$, because $\cB$ is invariant under $\pi$,
  $Q^\even_k(v) = Q^\odd_k(\pi(v))$.  For all $v \in V_\even$ such
  that $v$ or $\pi(v)$ is in $\inte_\odd\gamma$, consider the matching
  $M$ given by all of the $\{v, \pi(v)\}$ pairs. Let $M' \subseteq M$
  be the pairs $\{v, \pi(v)\}$ where exactly one of $v$ or $\pi(v)$ is
  in $\inte_\odd\gamma$.  For all edges in $M \setminus M'$ the terms
  for $v$ and $\pi(v)$ in the following equation cancel, so
  \begin{align*}
    \left|\sum_{v \in \inte_\odd \gamma \cap V_\even} Q^\even_k(v) - \sum_{v' \in \inte_\odd \gamma \cap V_\odd} Q^\odd_k(v') \right|
    & = \left|\sum_{\{v, \pi(v)\} \in M'} Q^\even_k(v) \right| \leq Q |\gamma|,
  \end{align*}
  where the final inequality follows because $M' \subseteq \gamma$.
\end{proof}

\begin{lemma}
  \label{lem:lambda_bound}
  Let $G$ be a transitive or matched automorphic graph satisfying
  Assumption~\ref{as:2}, and suppose the cycle basis is automorphism
  invariant.  There is a $\lambda_{\star}(D,\Delta)$ such that if
  $\lambda\geq\lambda_{\star}$, then for any contour $\gamma$,
  \begin{equation}
    |\widetilde{w}_\gamma | \leq \left(e^{3} \lambda^{-1/\Delta} \right)^{|\gamma|}.
  \end{equation}
\end{lemma}
\begin{proof}
  The proof is essentially the same for the vertex transitive and
  matched automorphic cases, so we consider both cases simultaneously.
  A contour $\gamma$ is \emph{thin} if
  $\cCb(\inte_{\even} \gamma)\cup\cCb(\inte_{\odd}\gamma)=\emptyset$.
  If $\gamma$ is thin, then $\Xi^\even_{\inte_{\odd} \gamma}$ and
  $\Xi^\odd_{\inte_{\odd} \gamma}$ are both equal to $1$. By
  Lemma~\ref{lem:b_transitive} (respectively
  Lemma~\ref{lem:b_matched}),
  \begin{equation}
    \widetilde{w}_\gamma = w_\gamma =
    \lambda^{-b(\gamma)} \leq \lambda^{-|\gamma|/\Delta}
    < \left(e^{3} \lambda^{-1/\Delta}  \right)^{|\gamma|}.
\end{equation}

We now proceed by induction on $|\inte \gamma|$. When
$|\inte \gamma| = 1$ the contour is necessarily thin (by
Lemma~\ref{lem:spo_order}) and the above argument applies. Let
$\gamma$ be a contour with $|\inte \gamma| = k+1$.  Without loss of
generality, assume $\gamma$ is even. Make the induction hypothesis
that for all contours $\gamma'$ with $|\inte \gamma'| \leq k$, 
\begin{equation}
  \widetilde{w}_{\gamma'} \leq \left(e^{3}\lambda^{-1/\Delta} \right)^{|\gamma'|}. 
\end{equation}

To advance the induction we will use Proposition~\ref{prop:FV729} with
$\cP = \cC^{\even}_{k}$ (respectively $\cC^{\odd}_{k}$) and
$\Lambda = \inte_{\odd}\gamma$. Note that all contours in
  $\cP_{\Lambda}$ satisfy the induction hypothesis by
  Lemma~\ref{lem:spo_order}. Using the
formulas~\eqref{eq:FV729Surface} and~\eqref{eq:FV729Surface-even} for
$\log \Xi^{\odd}_{\inte_{\odd}\gamma}$ and
$\log \Xi^{\even}_{\inte_{\odd}\gamma}$ yields
\begin{equation}
  \label{eq:XiDiffTM}
  |\log \Xi^{\odd}_{\inte_{\odd}\gamma}-\log
  \Xi^{\even}_{\inte_{\odd}\gamma}| =
  |\sum_{v\in V_{\odd}\cap \inte_{\odd}\gamma} Q^{\odd}_{k}(v) -
  S_{\odd}(\gamma)
  - \sum_{v\in V_{\even}\cap \inte_{\odd}\gamma} Q^{\even}_{k}(v) +
  S_{\even}(\gamma)|,
\end{equation}
and each term appearing in these formula is defined by a convergent
series by Proposition~\ref{prop:FV729}, provided
  $\lam\geq \lam_{\star}(D,\Delta)$ has been taken large
  enough. Hence by Lemma~\ref{lem:volume-cancel-transitive}
(respectively Lemma~\ref{lem:volume-cancel-matched}) combined with
Proposition~\ref{prop:FV729}, 
\begin{equation}
  \label{eq:XiDiffTM2}
  |\log \Xi^{\odd}_{\inte_{\odd}\gamma}-\log
  \Xi^{\even}_{\inte_{\odd}\gamma}| \leq \eta |\gamma| + 2|\gamma|,
\end{equation}
for some $\eta<1$ (after potentially increasing $\lam_{\star}$). Hence
\begin{equation}
  |\widetilde{w}_\gamma | = \left| \lambda^{-b(\gamma)}
    \frac{\Xi^\odd_{\inte_{\odd} \gamma}}{\Xi^\even_{\inte_{\odd}
        \gamma}}\right| \leq \lam^{-b(\gamma)}e^{3|\gamma|}\leq
  (e^{3}\lam^{-\frac{1}{\Delta}})^{|\gamma|},
\end{equation}
where we have used Lemma~\ref{lem:b_transitive} (respectively
Lemma~\ref{lem:b_matched}) to estimate
$b(\gamma)$. For $\gamma$ odd the same argument applies (up to
notational changes). This completes the proof.
\end{proof}

We can now summarize the main result of this section, which follows
from Lemmas~\ref{lem:lambda_bound} and~\ref{lem:BAI}.
\begin{prop}
  \label{prop:TMAce}
  Let $G$ be a graph satisfying Assumption~\ref{as:2}, and suppose $G$
  is transitive or matched automorphic.  There is a
  $\lam_{\star}(D,\Delta)$ such that if $\lam\geq\lam_{\star}$, then
  the cluster expansions for $\log \Xi_{\Lambda}^{\even}$ and
  $\log \Xi_{\Lambda}^{\odd}$ converge. Moreover,
  $\widetilde w_{\gamma}\leq
  (e^{3}\lam^{-\frac{1}{\Delta}})^{|\gamma|} $ for all contours
  $\gamma$.
\end{prop}

\subsection{Vertex Transitive within a Class}
\label{sec:convg_trans_within_class}

In this section we will make some further assumptions on $G$. 
\begin{assumption}
  \label{as:3}
  Assume $G$ is infinite, one-ended, bipartite, vertex transitive
  within each parity class, has a $D$-bounded cycle basis $\cB$, and
  has at most polynomial volume growth.
\end{assumption}
Note that under Assumption~\ref{as:3}, Lemma~\ref{lem:iso} implies
$\Phi_{G}(t)\geq \Ciso t^{-1/2}$ for some $\Ciso>0$. Applying this to
$\inte\gamma$ for any contour $\gamma$ yields the isoperimetric
inequality
\begin{equation}
  \label{eq:bireg-iso}
  |\gamma| \geq \Ciso |\inte\gamma|^{\frac12}
\end{equation}
which will be crucial. Recall that in the setting of
Assumption~\ref{as:3} we have activities $\lam_{\even}$ and
$\lam_{\odd}$ on even and odd vertices, and we parametrize
$\lam_{\odd}$ by introducing $\proplam$ and setting
\begin{equation}
  \label{eq:lamolam}
  \lam_{\odd} = \proplam
  \lam_{\even}^{\frac{\Delta_{\odd}}{\Delta_{\even}}}, \qquad \proplam \in
  \left(\frac12, 2 \right).
\end{equation}
We write $\cU$ for the set of $\lam_{\odd}$ arising
in this parametrization. The important property of the choice
$\proplam\in\left(\frac12,2\right)$ is that this bounded interval
contains $\proplam=1$. 

Unlike in Section~\ref{sec:conv_trans_autom}, there is no \emph{a
  priori} cancellation of volume factors in ratios
$\Xi^{\even}_{\Lam} / \Xi^{\odd}_{\Lam}$ of partition
functions. Instead we control this ratio by an inductive argument
using truncated weights.  This strategy, originally due to
Zahdradn\'{i}k~\cite{Zahradnik}, is now standard. In particular, it
has been clearly exposited
in~\cite[Chapter~7]{friedli2017statistical}, and we follow this
reference closely. Before giving some intuition for the approach, several
definitions are needed.

Define a contour $\gamma$ to be of \emph{class $n\geq 1$} if
$|\inte\gamma|=n$. We write $\cC_{n}$ for the subset of contours of
class $n$, $\cC^{x}_{n}\subset \cC_{n}$ for those of type
$x\in \{\even,\odd\}$, and define
$\cC_{\leq n}=\cup_{k=1}^{n}\cC_{k}$,
$\cC^{x}_{\leq n}=\cup_{k=1}^{n}\cC^{x}_{k}$.  For
$\gamma\in\cC^{x}_{1}$ of class one, the smallest possible class,
define the truncated weight $\widehat w_{\gamma}$ by
\begin{equation}
  \label{eq:tw1}
  \widehat w_{\gamma}=\widetilde w_{\gamma} = \lam_{x}^{-1},
\end{equation}
where the second equality holds as $\widetilde w_{\gamma}=w_{\gamma} =
\lam_{x}^{-|\gamma|/\Delta_{x}} = \lam_{x}^{-1}$.

To advance the definition of truncated weights to contours of classes
larger than one, introduce a cutoff parameter
\begin{equation}
  \label{eq:cutoff}
  \kappa = \frac{ \Ciso \log\lam_{\even}}{8\Delta_{\even}}, 
\end{equation}
and let $\chi\colon\R\to [0,1]$ be a $C^{1}$ cutoff function
satisfying $\chi(s)=1$ if $s\leq \kappa$ and $\chi(s)=0$ if
$s\geq 2\kappa$.  We will use below that $\norm{\chi'}_{\infty} < \infty$ since $\chi '$ is continuous and compactly supported. Suppose that truncated weights $\widehat w_{\gamma}$
have been defined for all contours $\gamma\in \cC_{\leq n}$. Given
this, recall $\widehat{w}(X)$ from~\eqref{eq:CEterms}, let $v\in V$,
and introduce the truncated free energies
\begin{equation}
  \label{eq:cpressure}
  \widehat \psi^{x}_{n} = q^{x} + \frac{\widehat
    Q^{x}_{n}}{\Delta_{x}},
  \qquad
  q^{x} = \frac{\log \lam_{x}}{\Delta_{x}}, \qquad \widehat
  Q^{x}_{n}(v) = \sum_{u\in N(v)} \mathop{\sum_{X\in\cX(\cC^{x}_{\leq
        n})}}_{\{u,v\}\in \overline X}\frac{1}{|\overline X|} \widehat{w}(X).
\end{equation}

Since $G$ is vertex transitive within a class, note that
$\widehat{Q}^{x}_{n}(v)$ is in fact independent of $v\in V$, and hence so is
$\widehat\psi^{x}_{n}$. We will shortly justify that $\widehat{Q}^{x}_{n}$ is in
fact finite, so that $\widehat\psi^{x}_{n}$ exists. Temporarily
granting this, we then define the truncated weights of contours of
class $n+1$ by 
\begin{equation}
  \label{eq:tw2}
  \widehat w_{\gamma} =
  \begin{cases}
    w_\gamma \chi\ob{
      \left( \widehat\psi^\odd_n - \widehat\psi^\even_n \right)
      \Delta_{\odd}|\inte_{\odd}\gamma\cap V_{\odd}|^{1/2}} 
      \frac{\Xi^\odd_{\inte_\odd \gamma}}{\Xi^\even_{\inte_\odd \gamma}}
      & \gamma\in \cC^{\even}_{n+1} \\
    w_\gamma \chi\ob{
      \left( \widehat\psi^\even_n - \widehat\psi^\odd_n \right)
      \Delta_{\even}|\inte_{\even}\gamma \cap V_{\even}|^{1/2}}
    \frac{\Xi^\even_{\inte_\even \gamma}}{\Xi^\odd_{\inte_\even \gamma}}
    & \gamma\in \cC^{\odd}_{n+1}
  \end{cases}.
\end{equation}

The next lemma partly explains the preceding definitions. The
definition of the truncated weights allows the hypothesis of the lemma
to be verified.
\begin{lemma}
  \label{lem:truncpress}
  Suppose the weights $\widehat w_{\gamma}$ satisfy~\eqref{eq:FV729Hw}
  for $\gamma\in \cC_{n}$, with $\tau\geq \tau_{1}$ from
  Proposition~\ref{prop:FV729}. Then for $x\in \{\even,\odd\}$,
  \begin{equation}
    \label{eq:truncpress}
    \widehat\psi^{x}_{n} = \lim_{k\to\infty} \frac{1}{|E(B_{k}(v))|}
    \log \lam_{x}^{|V_{x}\cap V(B_{k}(v))|} \widehat{\Xi}_{B_{k}(v)}^{x,n}
  \end{equation}
  where $\widehat{\Xi}_{\Lam}^{x,n}$ is the partition function with
  polymer set $\overline{\cC}^{x}_{\leq n}(\Lam)$.  In particular,
  $\widehat\psi^{x}_{n}-q^{x}\geq 0$.
\end{lemma}
\begin{proof}
  We first prove~\eqref{eq:truncpress} by
  using~\eqref{eq:FV729Surface-even} and
  Proposition~\ref{prop:FV729}. By Lemma~\ref{lem:vgrowth} the ratio
  of the boundary term to $|E(B_{k}(v))|$ vanishes as
  $k\to\infty$. The claim then follows since $Q(v)$ is independent of
  $v$ since the limiting ratio of $|V_{x}\cap V(B_{k}(v))|$ to
  $|E(B_{k}(v))|$ is $\Delta_{x}^{-1}$. The final claim
  $\widehat\psi^{x}_{n}-q^{x}\geq 0$ follows since
  $\widehat{\Xi}_{B_{k}(v)}^{x,n}\geq 1$.
\end{proof}

To gain some intuition, recall that
$\lam_{\odd} = \proplam \lam_{\even}^{\Delta_{\odd}/\Delta_{\even}}$,
and note that by Lemma~\ref{lem:VTPC},
\begin{equation}
  \label{eq:vtpc-2}
  \widetilde w_{\gamma} = \lam_{\even}^{-|\gamma|/\Delta_{\even}}
  \frac{\proplam^{|\inte_{\odd}\gamma\cap
      V_{\odd}|}\Xi^{\odd}_{\inte_{\odd}\gamma}}{\Xi^{\even}_{\inte_{\odd}\gamma}},
  \qquad \text{$\gamma$ even}.
\end{equation}
A similar expression can be written for odd contours. The first term
in~\eqref{eq:vtpc-2} provides decay in the size of $\gamma$, but it is
possible that the ratio in~\eqref{eq:vtpc-2} overwhelms this, making
$\widetilde w_{\gamma}$ large. If this occurs, it indicates that the
system does not want to be in the even phase, as even contours
represent deviations from the even phase. Unfortunately it is not
possible to make this precise by using a cluster expansion argument,
as the cluster expansion relies on weights being small.

Truncating weights allows for the cluster expansion to be
applied. Intuitively, the preceding paragraph suggests that truncation
should only change the weights of at most one type (even or odd) of
contour. This intuition is self-consistent: assuming the odd phase is
dominant, we expect that large even contours are rare, as this would
require the existence of a large even-occupied region. Such a contour
would necessarily live inside of a large odd contour, which is itself
rare, and this is captured by the truncated weights. Truncation should
thus have little effect; in particular the weights of at least one
type of contour should be unchanged. This ultimately enables the
determination of the phase diagram.

For a succinct summary of how this strategy can be carried out in the
context of spin systems on $\mathbb{Z}^{d}$
see~\cite{kotecky2006pirogov}, and for a more detailed discussion,
see~\cite[Chapter~7]{friedli2017statistical}. In the present context
the main technical conclusion is the following proposition.  Set
$\widehat\psi_{n} = \max \{
\widehat\psi^{\even}_{n},\widehat\psi^{\odd}_{n}\}$, and set
\begin{equation}
  \label{eq:adefinition}
  a^{x}_{n} = \widehat\psi_{n} - \widehat\psi^{x}_{n}, \qquad x\in \{\even,\odd\}.
\end{equation}
Note that at least one of $a^{\even}_{n}$ and $a^{\odd}_{n}$ is
zero. Moreover, for all $\gamma \in \cC^{\even}_{n+1}$,
$a_n^{\even} \Delta_{\odd} | \inte\gamma|^{1/2} \leq
\kappa$ implies $\widehat{w}_\gamma = \widetilde{w}_\gamma$, and
analogously for odd contours.  
Recall that
$\Delta = \max \{ \Delta_{\even},\Delta_{\odd}\}$, and recall $\Ciso$
from~\eqref{eq:bireg-iso}.
\begin{prop}
  \label{prop:FV734}
  Suppose $G$ satisfies Assumption~\ref{as:3}.
   For any $\tau$ sufficiently large, there exists $\lambda_\star = \lambda_\star(\Ciso, D, \Delta,\tau)$ and an increasing
  sequence $(c_{n})_{n\geq 0}$ with
  $2\leq c_{n}\uparrow c_\infty \leq 3$ so that for  $\lambda_\even > \lambda_\star$, $\lam_{\odd}\in\cU$, and all $n \geq 1$, the following statements hold: 
  \begin{enumerate}
  \item For $\gamma\in\cC_{\leq n}$,
    \begin{equation}
      \label{eq:weightbound}
      \widehat{w}_\gamma \leq e^{-\tau |\gamma|}
    \end{equation}
    and if $\gamma\in\cC_{\leq n}^{x}$, $x\in \{\even,\odd\}$,
    \begin{equation}
      \label{eq:truncequal}
      a^x_n \Delta_x|\inte\gamma|^{1/2} \leq \kappa/2 \text{ implies }
      \widehat{w}_\gamma = \widetilde{w}_\gamma. 
    \end{equation}
    Moreover, $\widehat{w}_\gamma$ is continuously differentiable in
    $\lam_{y}$, $y\in \{\even,\odd\}$, and
    \begin{equation}
      \label{eq:weightderivativebound}
      \left| \frac{d \widehat{w}_\gamma}{d
          \lambda_{y}}\right| \leq \frac{R
        |\gamma|^2}{\lam_{y}} e^{-\tau |\gamma|},
    \end{equation}
    where $R =4(1+\norm{\chi'}_{\infty})$, with $\chi$ the $C^1$ cutoff function defined above.
  \item Assume $\Lambda=\inte\gamma$ for $\gamma\in\cC_{\leq n}$.
    Then for $x,y\in \{\even,\odd\}$,
  \begin{align}
    \label{eq:Zind1}
    Z^x_{\Lam} &\leq \exp\left( \widehat{\psi}_n \Delta_{x} |\Lambda \cap V_{x}| + c_n |\gamma| \right) 
    \\
  \label{eq:Zind2}
  \abs{\frac{d Z^{y}_{\Lam}}{d\lam_{x}}} &\leq \frac{|\Lam\cap V_{x}|}{\lam_{x}} \exp
    (\widehat \psi_{n} \Delta_{y} |\Lam\cap V_{y}| +
    c_{n}|\gamma|).
  \end{align}
\end{enumerate}
\end{prop}
We have
  deliberately formulated the result to look
  like~\cite[Proposition~7.34]{friedli2017statistical}.
Given the notation and setup above, the proof of
  Proposition~\ref{prop:FV734} involves no new ideas and follows this reference.   
  To avoid reproducing a somewhat lengthy proof with no new
  insights, below we simply indicate the minor differences that arise
  in establishing Proposition~\ref{prop:FV734}.

\begin{proof}[Sketch of Proof of Proposition~\ref{prop:FV734}]
  The proof of Proposition~\ref{prop:FV734} follows the proof
  of~\cite[Proposition~7.34]{friedli2017statistical} closely. The
  changes needed reflect (i) our general combinatorial framework
  compared to the $\Z^{d}$-specific construction
  in~\cite{friedli2017statistical} and (ii) the slightly different
  nature of the hard-core model compared to the Blume-Capel model. The
  second point arises as the hard-core model has activity parameters
  $\lam_{\odd}$ and $\lam_{\even}$ as opposed to a large parameter
  $\beta$ that appears in front of a Hamiltonian. This explains why we
  obtain factors $\lam_{x}^{-1}$ in derivative estimates -- a
  derivative decreases the power of the activity parameters by one.

  Our combinatorial framework results in and induction with the base
  case considering contours of class one. This is slightly different
  than in~\cite{friedli2017statistical}, where the base case considers
  contours with empty interiors. All quantities (weights and partition
  functions) in our base case can be explicitly computed, e.g., the
  weights of the smallest contours are $\lam_{\even}^{-1}$ and
  $\lam_{\odd}^{-1}$. These explicit formulas make the verifications
  straightforward.

  The induction step is essentially the same as
  in~\cite{friedli2017statistical}. The definition of the truncated
  weights is well-defined by Lemma~\ref{lem:truncpress}. The key
  analytical tools~\cite[Lemmas~7.29 and~7.31]{friedli2017statistical}
  concerning the cluster expansion were established in
  Proposition~\ref{prop:FV729}. Three further points are worth
  remarking on. First, our hypotheses ensure that the worst possible
  isoperimetric behavior is the same (up to constants) as the
  isoperimetric behavior of $\mathbb{Z}^{2}$,
  see~\eqref{eq:bireg-iso}. Isoperimetric arguments are thus
  essentially identical to those
  in~\cite{friedli2017statistical}. Second, our contour weights are
  formulated somewhat differently than those
  in~\cite{friedli2017statistical}. In particular,
  $\widetilde w_{\gamma}$ can be expressed solely as a ratio of
  partition functions. This only simplifies matters. The boundary cost
  of contours is encoded in the comparison of $b_{\even}(\gamma)$ and
  $b_{\odd}(\gamma)$; the concrete statement that enables this is
  Lemma~\ref{lem:VTPC}. Third, a key aspect of the induction is the
  definition of the truncated weights, in particular the choice of
  $\kappa$. We have made essentially the same definition as
  in~\cite{friedli2017statistical}, up to some numerical factors that
  arise from our combinatorial framework.
\end{proof}

\section{Applications}
\label{sec:applications}

This section completes the proofs of our main theorems.
Theorem~\ref{thm:biregular} is proved in Section~\ref{sec:phases},
apart from establishing phase coexistence. Phase coexistence and
Theorems~\ref{thm:transitive} and~\ref{thm:matched} are established in
Section~\ref{sec:phase-coexistence}. Lastly
Section~\ref{sec:algorithms} outlines the proof of
Theorem~\ref{thm:expansion}.

\subsection{Phase Diagram}
\label{sec:phases}

Proposition~\ref{prop:FV734} is uniform in $n$, which allows us to
define 
\begin{equation}
  \label{eq:limf}
  \widehat\psi^{\even} = \lim_{n\to\infty}\widehat\psi^{\even}_{n}, \qquad
    \widehat\psi^{\odd} = \lim_{n\to\infty}\widehat\psi^{\odd}_{n}.
\end{equation}
and these limits exist as the truncated weights satisfy the hypotheses
of Proposition~\ref{prop:FV729} when $\lam_{\even}\geq\lam_{\star}$
and $\lam_{\odd}\in\cU$. In particular, we have
\begin{equation}
  \label{eq:truncfree}
  \widehat\psi = \lim_{n\to\infty}\widehat\psi_{n},
\end{equation}
and importantly, $\widehat\psi$ is the true free energy of the hard-core model on a
graph $G$, as the next proposition formalizes.

\begin{prop}
  \label{prop:pressure}
  Under Assumption~\ref{as:3}, if $\lam_{\even}\geq\lam_{\star}$ and
  $\lam_{\odd}\in\cU$, then for $v\in V$,
  \begin{equation}
    \label{eq:pressure}
    \lim_{n\to\infty} \frac{ \log Z^{\even}_{B_{n}(v)}}{|E(B_{n}(v))|} = \widehat\psi.
  \end{equation}
  The same limit is obtained with odd boundary conditions, or with any
  boundary condition in the sense of Section~\ref{sec:md-infinite}. 
\end{prop}
\begin{proof}
  Independence from the boundary condition follows from
  Lemma~\ref{lem:vgrowth} and assumption of $G$ having a bounded cycle basis,
  c.f.\ the proof of Lemma~\ref{lem:bcequiv}.

  Restricting to even boundary conditions,
  Proposition~\ref{prop:FV734} part 2.\ yields an upper bound of
  $\widehat\psi$. A matching lower bound is obtained by using that the
  quantity $\lam_{x}^{|V_{x}\cap V(B_{k}(v))|} \Xi_{B_{k}(v)}^{x,n}$
  used in defining $\widehat\psi^{x}_{n}$ in
  Lemma~\ref{lem:truncpress} is a lower bound for
  $Z^{\even}_{B_{k}(v)}$; this holds as only a subset of polymers are
  allowed compared to the full polymer representation of
  $Z^{\even}_{B_{k}(v)}$.
\end{proof}

\begin{lemma}
  \label{lem:coex-point}
  Suppose $G$ satisfies Assumption~\ref{as:3}.  If
  $\lam_{\even}\geq \lam_{\star}(\Ciso,D,\Delta)$, there is a unique
  $\lam_{\odd,c}\in\cU$ such that   $\widehat\psi^{\even}=\widehat\psi^{\odd}$. 
\end{lemma}
\begin{proof}
  By Propositions~\ref{prop:FV734}, if $\lam_{\star}$ is large enough
  the hypotheses of Proposition~\ref{prop:FV729} are satisfied, and
  hence $\widehat\psi^{\even}$ and $\widehat\psi^{\odd}$ are
  continuously differentiable functions of $\lam_{\even}$ and of
  $\lam_{\odd}$ for $\lam_{\odd}\in \cU$, and hence also of $\proplam$
  when we write
  $\lam_{\odd}=\proplam
  \lam_{\even}^{\Delta_{\odd}/\Delta_{\even}}$. Note that
  $\widehat\psi^{\odd}=\widehat\psi^{\even}$ if and only if
  \begin{equation}
    \label{eq:coex2}
    \widehat{\psi}^{\odd} - \widehat{\psi}^{\even} =
    \frac{\log\proplam}{\Delta_{\odd}} - \frac{\widehat
      Q^{\even}}{\Delta_{\even}} + \frac{\widehat
      Q^{\odd}}{\Delta_{\odd}} = 0. 
  \end{equation}
  Moreover, Proposition~\ref{prop:FV729} yields
  \begin{align}
    \label{eq:coex1}
    |\widehat\psi^{\even}-\frac{\log \lam_{\even}}{\Delta_{\even}}|
    &= | \frac{\widehat Q^{\even}}{\Delta_{\even}} | \leq \eta \\
    \label{eq:coex11}
    |\widehat\psi^{\odd}-\frac{\log \lam_{\even}}{\Delta_{\even}} -
    \frac{\log \proplam}{\Delta_{\odd}}| 
    &= | \frac{\widehat Q^{\odd}}{\Delta_{\odd}} | \leq \eta
  \end{align}
  where we recall $\eta=e^{-\tau/3}$, and $\tau\uparrow\infty$ as
  $\lam_{\star}\uparrow\infty$. In the subsequent steps of the proof we
  may increase $\lam_{\star}$ without explicitly saying so, if necessary.
  
A solution $\proplam_{c}\in (\frac12, 2)$ to~\eqref{eq:coex2} exists
by the intermediate value theorem, as $\log\proplam$,
$\widehat Q^{\even}$, and $\widehat Q^{\odd}$ are all continuous
functions of $\proplam$, and the quantity is negative (positive) for
$\rho$ sufficiently close to $1/2$ ($2$). This solution $\proplam_{c}$
is unique, as the derivative of the left-hand side of~\eqref{eq:coex2}
with respect to $\proplam$ is uniformly positive for
$\proplam\in \cU$. This is because there is an $K>0$ such that both
$\widehat Q^{\even}$ and $\widehat Q^{\odd}$ have derivatives
uniformly bounded by $K\eta$ by Propositions~\ref{prop:FV729}
and~\ref{prop:FV734}, as the factor arising from
  differentiating $\lam_{\odd}$ as a function of $\rho$ is compensated
  for by the factor $\lam_{\odd}^{-1}$
  in~\eqref{eq:weightderivativebound}.
\end{proof}

\begin{proof}[Proof of Theorem~\ref{thm:biregular}, parts 1.\ and 2.]
  As in the proof of Lemma~\ref{lem:coex-point}, we argue using
  $\proplam$ as a parameter. First observe that the free energy
  $\widehat\psi$ exists by Proposition~\ref{prop:pressure}, which
  shows $\widehat\psi = f_{G}$ as defined by~\eqref{eq:fg}.

  For part 1., note that by~\eqref{eq:coex2}, the definition of
  $\widehat\psi$, and the monotonicity of $\log\proplam$, if
  $\proplam<\proplam_{c}$ then $\widehat\psi = \widehat\psi^{\even}$,
  and if $\proplam>\proplam_{c}$ then
  $\widehat\psi = \widehat\psi^{\odd}$.  This enables the computation
  of the derivative of $\widehat\psi$ with respect to $\proplam$. That
  is, the derivatives of $\widehat\psi^{\even}$ and
  $\widehat\psi^{\odd}$ with respect to $\proplam$ are
  $\frac{1}{\Delta_{\even}}\frac{\textrm{d}\widehat
    Q^{\even}}{\textrm{d}\proplam}$ and
  $\frac{1}{\proplam
    \Delta_{\odd}}+\frac{1}{\Delta_{\odd}}\frac{\textrm{d}\widehat
    Q^{\odd}}{\textrm{d}\proplam}$,
  respectively. Propositions~\ref{prop:FV729} and~\ref{prop:FV734}
  imply that for $\proplam\in (\frac12,2)$, these derivatives of
  $Q^{\even}$ and $Q^{\odd}$ exist and are continuous. This
    proves that continuous differentiability may fail only at pairs
    $(\lam_{\even},\lam_{\odd,c}(G,\lam_{\even}))$.

  For part 2., observe that re-arranging~\eqref{eq:coex2} yields a
  formula for $\lam_{\odd,c}$: it solves
  \begin{equation}
    \label{eq:coex-point}
    \frac{\log \lam_{\odd,c}}{\Delta_{\odd}} = \frac{\log
      \lam_{\even}}{\Delta_{\even}} + \frac{\widehat
    Q^{\even}(\lam_{\even},\lam_{\odd,c})}{\Delta_{\even}}-\frac{\widehat
    Q^{\odd}(\lam_{\even},\lam_{\odd,c})}{\Delta_{\odd}},
  \end{equation}
  and the stated estimate follows by using the bounds~\eqref{eq:coex1}
  and~\eqref{eq:coex11} as $\eta\downarrow 0$ as
  $\lam_{\star}\uparrow \infty$.
\end{proof}

\begin{remark}
  \label{rem:rhomb}
  Recalling Example~\ref{ex:biregular}, the more precise formula for
  coexistence follows from~\eqref{eq:coex-point}, as the leading
  contributions to $Q^{\even}$ and $Q^{\odd}$ are
  $\lam_{\even}^{-1}$ and $\lam_{\odd}^{-1}$. Taking these terms into
  account shows the (leading) effect of $\Delta_{\even}$ and
  $\Delta_{\odd}$. The effect of the geometry of the graph $G$ beyond
  $\Delta_{\odd}$ and $\Delta_{\even}$ is encoded in higher-order terms.
\end{remark}

To prove Theorem~\ref{thm:biregular}, part 3.\ we use a result of
  van den Berg and
  Steif~\cite[Proposition~4.6]{van1994percolation}. They state their
  result in the setting of $\Z^{d}$, but their proof extends and yields the
  following. 
  \begin{lemma}
    \label{lem:vdBS}
    Suppose $G$ is infinite and bipartite, that
    $f_{G}(\lam_{\even},\lam_{\odd})$ exists, is independent of the
    boundary conditions chosen in its definition, and is
    differentiable in both variables at
    $(\lam_{\even},\lam_{\odd})$. Then there is a unique Gibbs measure
    at $(\lam_{\even},\lam_{\odd})$.
  \end{lemma}
\begin{proof}[Proof of Theorem~\ref{thm:biregular}, part 3.]
  The conditions of Lemma~\ref{lem:vdBS} are the conclusion of
  Theorem~\ref{thm:biregular}, 1., which we have already verified.
\end{proof}

The proof of Lemma~\ref{lem:vdBS} uses monotonicity properties of the
hard-core model on bipartite graphs, and hence so does our proof of
Theorem~\ref{thm:biregular} part 3. This use of monotonicity could be
avoided by arguing directly in terms of contours; we have not done so
for the sake of efficiency.

\subsection{Phase Coexistence}
\label{sec:phase-coexistence}

In this section we deduce phase coexistence from our earlier results
by a standard (Peierls-type) argument. We consider $\Lambda_{n}\uparrow G$
(see Section~\ref{sec:md-infinite}) with the additional property that
each $\Lambda_{n}$ arises as the interior of a contour. Such a
sequence exists by the bijection described in
Proposition~\ref{lem:bijection}.

Recall that $\mu^{\even}_{\Lambda}$ denotes the hard-core model on
$\Lambda$ with even boundary conditions, and that $c$ is a constant in
the lower bound on the isoperimetric profile in the hypotheses of
Theorem~\ref{thm:matched}.
\begin{prop}
  \label{prop:pt}
  Suppose $G$ is an infinite bipartite graph satisfying the hypotheses
  of Theorem~\ref{thm:transitive} or~\ref{thm:matched}. Let $v_{\odd}$
  and $v_{\even}$ denote fixed vertices in $V_{\odd}$ and
  $V_{\even}$. If $\lambda\geq\lam_{\star}(\Ciso,D,\Delta)$, then
  \begin{equation}
    \label{eq:pt}
    \lim_{n\to\infty}  \mu^{\even}_{\Lambda_{n}}[\textrm{$v_{\even}$
      is occupied}] >\frac{1}{2}, \quad \text{and} \quad 
    \lim_{n\to\infty}  \mu^{\odd}_{\Lambda_{n}}[\textrm{$v_{\odd}$ is
      occupied}] >\frac{1}{2}  \,.
  \end{equation}
\end{prop}

The previous proposition yields~Theorem~\ref{thm:transitive} and
Theorem~\ref{thm:matched}.
\begin{proof}[Proofs of Theorems~\ref{thm:transitive}
  and~\ref{thm:matched}]
  Fix $\lam$ large according to Proposition~\ref{prop:pt}, and let
  $v_{\even}$ and $v_{\odd}$ be adjacent vertices. Recalling
  Section~\ref{sec:md-infinite}, the limiting Gibbs measure
  $\mu^{\even} = \lim_{n\to\infty}\mu^{\even}_{\Lambda_{n}}$ exists by
  Lemma~\ref{lem:bcequiv}. The same is true for $\mu^{\odd}$. Since
  $\mu^{\even}[\textrm{$v_\even$ occupied}]>\frac12$,
  $\mu^{\even}[\textrm{$v_{\odd}$ occupied}]<\frac12$. Hence
  $\mu^{\even}$ and $\mu^{\odd}$ have distinct marginals, and are
  therefore distinct measures.
\end{proof}

\begin{proof}[Proof of Proposition~\ref{prop:pt}]
  We consider $\mu^{\even}$ (the argument for $\mu^{\odd}$ is
  analogous). Let $v=v_{\even}$. For $n$ large enough, $v$ is occupied
  unless there is a contour that separates $v$ from
  $\binf \Lambda_{n}$. By Proposition~\ref{prop:TMAce}, the
  probability of a contour of size $n$ is $\exp(-c_{1}(\lambda) n)$
  for some $c_{1}(\lambda)>0$ with $c_{1}(\lambda)\to\infty$ as
  $\lambda\to\infty$. By Lemma~\ref{lem:enum_contours}, the number of
  contours of size $n$ containing a fixed edge $e$ is at most
  $\exp(c_{2}n)$ for some absolute $c_{2}$ depending only on
  $G$. Finally, if a contour contains an edge at graph distance $k$
  from a vertex $v$, and $v$ is contained in the interior of the
  contour, then the contour must contain at least $\Ciso \log (k+1)$
  edges since $\Phi_{G}(t) \geq c\log (t+1) / t$. Note that
  while this is an assumption in the context of
  Theorem~\ref{thm:matched}, it also holds in the context of
  Theorem~\ref{thm:transitive} by Lemma~\ref{lem:iso}.

  Fix a self-avoiding path from $v$ to $\binf \Lambda_{n}$. If a
  contour contains $v$ in its interior, it must contain an edge in
  this path. Hence
  \begin{equation*}
    \mu^{\even}_{\Lambda_{n}}[\textrm{$v$ is
    occupied}] \geq 1-\sum_{k\geq 1} \sum_{j\geq \Ciso \log (k+1) }
    \exp((c_{2}-c_{1}(\lambda))j),
  \end{equation*}
  by using the observation that if $v$ is not occupied, a separating
  contour must exist, and then applying a union bound; we have dropped
  the condition that the contours are contained in $\Lambda_{n}$ to
  obtain an upper bound. For $\lambda$ large enough the final sum is
  as small as desired, which gives the conclusion.
\end{proof}

\begin{proof}[Proof of Theorem~\ref{thm:biregular}, part 4.]
  We first prove phase coexistence occurs. By
  Lemma~\ref{lem:coex-point} and Proposition~\ref{prop:FV734},
  $\widehat w_{\gamma}=\widetilde w_{\gamma}$ for all contours
  $\gamma$ when $\lam_{\odd}=\lam_{\odd,c}$.  By following the
  argument for Proposition~\ref{prop:pt} we then obtain \eqref{eq:pt}
  (the required lower bound on the isoperimetric profile holds by
  Lemma~\ref{lem:iso}). Phase coexistence then follows as in the proof
  of Theorems~\ref{thm:transitive} and~\ref{thm:matched}.

  The last conclusion follows, as if $f_{G}$ was differentiable at
  such pairs, then Lemma~\ref{lem:vdBS} would imply uniqueness occurs;
  recall that the hypothesis of Lemma~\ref{lem:vdBS} have been
  verified in by the proof of Theorem~\ref{thm:biregular}, part 1.
\end{proof}

\subsection{Algorithms}
\label{sec:algorithms}

In this section we indicate how to prove Theorem~\ref{thm:expansion}
by using the method employed in~\cite{helmuth2020algorithmic,BCHPT}.
\begin{proof}[Proof of Theorem~\ref{thm:expansion}]
  Note that Lemma~\ref{lem:allowed} (combined with Lemma~\ref{lem:int_ind}
  and Proposition~\ref{lem:bijection}) ensures that graphs in
  $\mathcal{H}$ allow a reformulation of the hard-core model on
  $H\in\mathcal{H}$ as a contour model. Given this, the argument is the same as
  in~\cite[Section~5.1 and Section~6]{BCHPT}. While the present
  context concerns more general graphs, the essential points are (i)
  there is an ordering of contours into \emph{levels} such that the
  weight of contours of level $k+1$ only depends on contours of level
  at most $k$ and (ii) that all contours of size $k$ can be enumerated
  in time exponential in $k$. The first of these facts was established
  in Section~\ref{sec:contour-polym-repr} using the ordering $\prec$
  from Section~\ref{sec:order}. The second fact is standard, as
  Lemma~\ref{lem:enum_contours} established that the contours we want
  to enumerate can be identified with connected induced subgraphs of a
  bounded-degree graph. The algorithmic enumeration of such objects in
  exponential time is well-known, see,
  e.g.,~\cite[Lemma~3.4]{patel2016deterministic}.
\end{proof}

\begin{remark}
  Theorem~\ref{thm:expansion} only concerns settings in which phase
  coexistence occurs. It is also possible to obtain algorithms
  when there is uniqueness of Gibbs measures and both stable and
  unstable ground states exist, see~\cite{BCHPT}.  This method could
  be used to develop algorithms in the context of
  Theorem~\ref{thm:biregular}; but we have chosen not to pursue this
  in the present paper.
\end{remark}

\section*{Acknowledgements}

This work was done as part of an AIM SQuaRE on `Connections Between
Computational and Physical Phase Transitions', and the authors thank
AIM for the generous support.  The authors also thank Antonio Blanca,
Tom Hutchcroft, Alexandre Stauffer, and Izabella Stuhl for many
helpful conversations on the topic.

S.\ Cannon is supported in part by NSF CCF-2104795.   W.\ Perkins is
supported in part by NSF DMS-2348743.

\end{document}